\documentclass[11pt,reqno]{amsart}

\usepackage[T1]{fontenc}
\usepackage{lmodern}
\usepackage{microtype}
\usepackage[margin=1.05in,headheight=14pt]{geometry}
\usepackage{amsmath,amssymb,mathtools}
\usepackage{booktabs,tabularx,array}
\usepackage{enumitem}
\usepackage[dvipsnames]{xcolor}
\usepackage[
  colorlinks=true,
  linkcolor=MidnightBlue,
  citecolor=BrickRed,
  urlcolor=RoyalBlue
]{hyperref}
\usepackage[nameinlink,noabbrev]{cleveref}

\usepackage{mathrsfs}

\usepackage{tikz-cd}

\setlist[itemize]{leftmargin=1.45em,itemsep=2pt,topsep=4pt}
\setlist[enumerate]{
  leftmargin=1.75em,
  itemsep=2pt,
  topsep=4pt,
  font=\normalfont
}

\newtheorem{theorem}{Theorem}[section]
\newtheorem{proposition}[theorem]{Proposition}
\newtheorem{lemma}[theorem]{Lemma}
\newtheorem{corollary}[theorem]{Corollary}
\theoremstyle{definition}
\newtheorem{definition}[theorem]{Definition}
\theoremstyle{remark}
\newtheorem{remark}[theorem]{Remark}
\newtheorem{example}[theorem]{Example}
\newtheorem{question}[theorem]{Question}

\newcommand{\cE}{\mathcal E}
\newcommand{\cF}{\mathcal F}
\newcommand{\cG}{\mathcal G}
\newcommand{\cH}{\mathcal H}
\newcommand{\cO}{\mathcal O}

\newcommand{\bR}{\mathbb R}
\newcommand{\sS}{\mathscr S}
\newcommand{\sV}{\mathscr V}
\newcommand{\Exc}{\operatorname{Exc}}
\newcommand{\Supp}{\operatorname{Supp}}
\newcommand{\Sing}{\operatorname{Sing}}
\newcommand{\coeff}{\operatorname{coeff}}
\newcommand{\vol}{\operatorname{vol}}
\newcommand{\tang}{\mathrm{tang}}

\newcommand{\red}{\mathrm{red}}
\newcommand{\mult}{\mathrm{mult}}
\newcommand{\Aut}{\mathrm{Aut}}
\newcommand{\Bir}{\mathrm{Bir}}

\title[Birational Automorphism Bounds via Pluricanonical Indices]
{Birational Automorphism Bounds for General-Type Foliations on Surfaces
via Pluricanonical Indices}

\author{Shi Xu}
\address{Yau Mathematical Sciences Center, Tsinghua University,
Beijing 100084, China}
\email{shixumath@163.com; xushiymsc@mail.tsinghua.edu.cn}
\date{\today}



\begin{document}

\begin{abstract}
Let $\mathcal{F}$ be a canonical foliation of general type on a smooth
projective surface $X$, and write
$\mathrm{vol}(\mathcal{F}):=\mathrm{vol}(K_{\mathcal{F}})$.  Let
$G\subseteq\operatorname{Bir}(X,\mathcal{F})$ be a finite subgroup, and let
$\mathcal{G}:=\mathcal{F}/G$ be the quotient foliation in the birational sense.
For a canonical foliation $\mathcal{H}$, define its $r$-th pluricanonical
index by
\[
  \delta_r(\mathcal{H})
  :=
  \min
  \bigl\{
    m\in\mathbb{Z}_{>0}
    \mid
    h^0(mK_{\mathcal{H}})\geq r
  \bigr\},
\]
where $\min\varnothing:=\infty$.  For an arbitrary foliation, these
indices are computed on any canonical birational model.

If $\kappa(\mathcal{G})\geq0$, we prove
\[
  |G|
  \leq
  \begin{cases}
    4\delta_1(\mathcal{G})\,\mathrm{vol}(\mathcal{F}),
      &\kappa(\mathcal{G})=0,\\[1mm]
    \displaystyle
    \frac{4}{3}\delta_2(\mathcal{G})\,\mathrm{vol}(\mathcal{F}),
      &\kappa(\mathcal{G})=1,\\[3mm]
    \delta_2(\mathcal{G})^2
    \bigl(1+\delta_2(\mathcal{G})\bigr)\,
    \mathrm{vol}(\mathcal{F}),
      &\kappa(\mathcal{G})=2.
  \end{cases}
\]
Since $\operatorname{Bir}(X,\mathcal{F})$ is finite, one may in particular
take $G=\operatorname{Bir}(X,\mathcal{F})$.  
When $\kappa(\mathcal{G})=0$ or $1$, the effective bounds
$\delta_1(\mathcal{G})\leq12$ and $\delta_2(\mathcal{G})\leq42$ give
\[
  |G|\leq48\,\mathrm{vol}(\mathcal{F})
  \qquad\text{and}\qquad
  |G|\leq56\,\mathrm{vol}(\mathcal{F}),
\]
respectively.

The main new ingredient is a cluster formula for adjoint volumes,
which yields index-dependent lower bounds for tangency-free
foliated surface pairs whose underlying foliation has Kodaira
dimension zero or one.
\end{abstract}

\maketitle

\setcounter{tocdepth}{1}

\tableofcontents

\section{Introduction}\label{sec:introduction}

Throughout the paper, all varieties are defined over $\mathbb C$.

\subsection*{Background and motivation}

A classical problem in birational geometry is to bound automorphism
groups in terms of canonical volumes.  For a minimal surface $S$ of
general type, Xiao proved the explicit linear estimate
\[
  |\operatorname{Aut}(S)|
  \leq
  42^2K_S^2
  =
  42^2\vol(K_S);
\]
see \cite{Xia94,Xia95}.  Since a minimal surface of general type is
the unique minimal model in its birational class, every birational
self-map of $S$ is biregular; see
\cite[Proposition~III.4.6]{BHPV04}.  Thus
$\operatorname{Bir}(S)=\operatorname{Aut}(S)$, and Xiao's estimate may
equivalently be stated for the birational automorphism group.

More generally, for every fixed dimension $n$, the theorem of
Hacon--McKernan--Xu gives a constant $c_n$ such that
\[
  |\operatorname{Bir}(V)|
  \leq
  c_n\vol(K_V)
\]
for every $n$-dimensional projective variety $V$ of general type;
see \cite{HMX13}.  These results are closely related to effective
pluricanonical maps and to the DCC for canonical volumes
\cite{HMX14}.

It is natural to ask for analogous bounds for foliations.  Pereira and
Fern\'andez S\'anchez proved that the group of birational
transformations preserving a foliation of general type is finite
\cite[Theorem~1]{PS02}.  Several quantitative results are also known
under additional assumptions.  
When $K_{\cF}$ is ample and
$\Aut(X,\cF)$ is finite, Corr\^ea and Fassarella obtained an explicit
upper bound for its order depending on $K_{\cF}^2$ and
$K_{\cF}\cdot K_X$ \cite[Theorem~4.2]{CF14}.  
Corr\^ea and Muniz obtained, under various geometric assumptions,
polynomial bounds for finite groups of automorphisms preserving $\cF$
in terms of the Chern numbers of $X$ and $\cF$
\cite[Sections~4--6]{CM19}.
More recently, Spicer and Svaldi proved that, for every sufficiently
small $\epsilon>0$, the order of $\Bir(X,\cF)$ is bounded by
$C(\epsilon)\vol(K_{\cF}+\epsilon K_X)$ for canonical foliations of
general type on smooth projective surfaces
\cite[Theorem~5.3]{SS23}.  
The preceding bounds involve either auxiliary intersection numbers or
the volume of an adjoint divisor.  This leads to the following
question, suggested in \cite[p.~48]{SS23}.

\begin{question}
\label{ques:volume-bound}

Does there exist a function
$f\colon\bR_{>0}\to\bR_{>0}$ such that
\[
  \#\Bir(X,\cF)
  \leq
  f\bigl(\vol(K_{\cF})\bigr)
\]
for every canonical foliation $\cF$ of general type on a smooth
projective surface $X$?
\end{question}
Motivated by Question~\ref{ques:volume-bound}, we establish bounds
linear in $\vol(K_{\cF})$ for finite subgroups
$G\subseteq\Bir(X,\cF)$ whenever $\kappa(\cF/G)\geq0$, with
coefficients expressed explicitly in terms of the relevant
pluricanonical indices of $\cF/G$.  When $\kappa(\cF/G)=0$ or $1$,
these indices are uniformly bounded, yielding uniform linear bounds
depending only on $\vol(K_{\cF})$.

\subsection*{Pluricanonical indices}

One classical approach to lower bounds for canonical volumes is
through effective control of pluricanonical systems.  
For canonical models of foliations of general type with fixed Hilbert
function, results on the birationality of pluricanonical maps were
obtained in \cite[Theorem~2]{HL21} and
\cite[Theorem~1.3]{Chen21}, with an effective refinement in
\cite[Corollary~1.7(2)]{LWX26}.  

In general, however, there is no universal integer
$m$ such that $|mK_{\cF}|$ defines a birational map for every canonical
foliation of general type, even in fixed dimension and rank
\cite{Lu25}.  This suggests retaining the weaker numerical data
recording the first occurrence of a nonzero pluricanonical section or
a pluricanonical pencil.

Low plurigenera play a central role in Chen--Chen's work on threefolds of general type \cite{CC10,CC15}. In analogy with their ps-index, we introduce the following pluricanonical indices for foliations.

Let $\cH$ be a
canonical foliation on a smooth projective surface $X$ and define, for
$r\geq1$,
\[
  \delta_r(\cH)
  :=
  \min
  \bigl\{
    m\in\mathbb Z_{>0}
    \mid
    h^0(X,mK_{\cH})\geq r
  \bigr\},
\]
where $\min\varnothing:=\infty$.
(For an arbitrary foliation, we define $\delta_r$ to be the
corresponding index of any canonical birational model.)
We call $\delta_r(\cH)$ the \emph{$r$-th pluricanonical index} of $\cH$.
Thus $\delta_2(\cH)$ is the foliated analogue of the ps-index of Chen--Chen.
When $\kappa(\cH)=0$, the index $\delta_1(\cH)$ coincides with
Pereira's height \cite{Per05}.
For an earlier
use of the analogous quantity for a normal surface---the least
integer $m>0$ such that $h^0(mK_X)\geq2$---in estimating
$K_X^2$, see Blache \cite{Bla95}.

For canonical foliations, these indices are birational invariants and
may therefore be computed on any smooth birational model.  Moreover,
\[
  \delta_1(\cH)<\infty
  \quad\Longleftrightarrow\quad
  \kappa(\cH)\geq0,
\qquad\qquad
  \delta_2(\cH)<\infty
  \quad\Longleftrightarrow\quad
  \kappa(\cH)\geq1.
\]
These two indices are therefore naturally adapted to foliations
of nonnegative Kodaira dimension.

In Kodaira dimension two, the ps-index already gives effective
numerical control of the canonical volume. 
In \cite[Corollary~1.2 and (1.6)]{Xu26}, the author proved that every canonical foliation
$\cH$ of general type satisfies
\begin{equation}\label{equ:volume-ps-index}
  \vol(\cH)
  :=
  \vol(K_{\cH})
  \geq
  \frac{1}{
    \delta_2(\cH)^2
    \bigl(1+\delta_2(\cH)\bigr)
  }.
\end{equation}
This estimate suggests that, in the absence of uniform birational
boundedness, the ps-index is a natural next invariant to consider:
it records only the first appearance of a pluricanonical pencil,
but is already strong enough to control the canonical volume.

\subsection*{Main results}

Our first main result gives birational automorphism bounds whenever
the quotient foliation has nonnegative Kodaira dimension.

\begin{theorem}[Birational automorphism bounds from pluricanonical indices]
\label{thm:main}

Let $X$ be a smooth projective surface and let $\cF$ be a canonical
foliation of general type on $X$.  Let
$G\subseteq\operatorname{Bir}(X,\cF)$ be a finite subgroup and let
$\cG:=\cF/G$ be the quotient foliation in the birational sense.
If $\kappa(\cG)\geq0$, then
\[
  |G|
  \leq
  \begin{cases}
    4\delta_1(\cG)\,\vol(\cF),
      &\kappa(\cG)=0,\\[1mm]
    \displaystyle
    \frac{4}{3}\delta_2(\cG)\,\vol(\cF),
      &\kappa(\cG)=1,\\[3mm]
    \delta_2(\cG)^2
    \bigl(1+\delta_2(\cG)\bigr)\,
    \vol(\cF),
      &\kappa(\cG)=2.
  \end{cases}
\]
\end{theorem}
Since $\operatorname{Bir}(X,\cF)$ is finite by
\cite[Theorem~1]{PS02}, one may take
$G=\operatorname{Bir}(X,\cF)$ in Theorem~\ref{thm:main}.

Theorem~\ref{thm:main} is obtained from the following adjoint volume
statement, which is the main technical result of the paper.

\begin{theorem}[Adjoint volume bounds from pluricanonical indices]
\label{thm:adjoint-volume-bounds}

Let $X$ be a smooth projective surface, let $\cF$ be a canonical
foliation on $X$, and let
\[
  \Delta=\sum_i a_iC_i,
  \qquad
  a_i\in\left[\frac12,1\right),
\]
where every $C_i$ is irreducible and non-$\cF$-invariant.  Assume that
\[
  \tang(\cF,\Delta_{\red})=0
\]
and that $K_{\cF}+\Delta$ is big.  If $\kappa(\cF)\geq0$, then
\[
  \vol(K_{\cF}+\Delta)
  \geq
  \begin{cases}
    \displaystyle
    \frac{1}{4\delta_1(\cF)}
    \geq
    \frac1{48},
      &\kappa(\cF)=0,\\[3mm]
    \displaystyle
    \frac{3}{4\delta_2(\cF)}
    \geq
    \frac1{56},
      &\kappa(\cF)=1,\\[3mm]
    \displaystyle
    \frac{1}{
      \delta_2(\cF)^2
      \bigl(1+\delta_2(\cF)\bigr)
    },
      &\kappa(\cF)=2.
  \end{cases}
\]

\end{theorem}

The cases $\kappa(\cF)=0$ and $\kappa(\cF)=1$ constitute the new
content of the adjoint volume theorem.  In these cases $K_{\cF}$ is
not big, so the bigness of $K_{\cF}+\Delta$ comes from the boundary,
and the lower bound is governed by $\delta_1(\cF)$ or
$\delta_2(\cF)$, respectively.  When $\kappa(\cF)=2$, monotonicity of
volume gives
\[
  \vol(K_{\cF}+\Delta)
  \geq
  \vol(K_{\cF}),
\]
and the asserted estimate follows from \eqref{equ:volume-ps-index}.  Thus the
tangency and coefficient assumptions are unnecessary in this case.
The uniform lower bounds $1/48$ and $1/56$ follow, respectively, from
Pereira's bound $\delta_1(\cF)\leq12$ when $\kappa(\cF)=0$
\cite[Theorem~1]{Per05} and the bound $\delta_2(\cF)\leq42$ when
$\kappa(\cF)=1$ \cite[Theorem~4.14]{CLNP22}.

Combining Theorem~\ref{thm:main} with these bounds gives the following
uniform estimates.

\begin{corollary}
\label{cor:uniform-automorphism-bounds}

Under the assumptions of Theorem~\ref{thm:main}, one has
\[
  |G|
  \leq
  \begin{cases}
    48\,\vol(\cF),
      &\kappa(\cG)=0,\\[1mm]
    56\,\vol(\cF),
      &\kappa(\cG)=1.
  \end{cases}
\]
Moreover,
\[
  \vol(\cF)
  \geq
  \begin{cases}
    \displaystyle\frac1{24},
      &\kappa(\cG)=0,\\[2mm]
    \displaystyle\frac1{28},
      &\kappa(\cG)=1.
  \end{cases}
\]

\end{corollary}

For the second assertion, note that if
$\kappa(\cG)\in\{0,1\}$, then $G$ is nontrivial; otherwise
$\cG=\cF$ would have Kodaira dimension two.  Hence $|G|\geq2$, and
the asserted volume bounds follow from the first assertion.

The volume estimates in
Corollary~\ref{cor:uniform-automorphism-bounds} fit into the broader
problem of understanding the possible volumes of foliations of
general type.  It is natural to ask whether, in fixed dimension and
rank, the set of volumes of canonical foliations of general type
satisfies the DCC and, in particular, admits a uniform positive lower
bound; see \cite[\S3]{Cas21} and \cite[Question~4]{HL21}.  The general
question remains open even for rank-one foliations on surfaces.

Using Noether-type inequalities, L{\"u} and Tan proved the sharp bound
$\vol(\cF)\geq\frac12$ under the assumption
$h^0(K_{\cF})\geq2$ \cite[Corollary~1.11]{LT24}.  
This is precisely
the case $\delta_2(\cF)=1$ of \eqref{equ:volume-ps-index}.
Han--Jiao--Li--Liu proved that the volume of a canonical
algebraically integrable foliation belongs to a discrete set depending
only on its rank and the volume of a general leaf
\cite[Theorem~1.2]{HJLL26}. 
For algebraically integrable foliations, the volume agrees with a
modular invariant of the inducing fibration \cite{LT24}; Liu and Tan
obtained genus-dependent lower bounds for this invariant, sharp in
genus two \cite{LT22}.
Related DCC and volume lower-bound
results under additional assumptions appear in
\cite{CPT25,Fan25,ST26,CLSS26,CZ26}.

\subsection*{Outline of the proof}

The case $\kappa(\cF)=2$ of
Theorem~\ref{thm:adjoint-volume-bounds} follows directly from
the monotonicity of volume and \eqref{equ:volume-ps-index}.  We
therefore focus on $\kappa(\cF)\in\{0,1\}$.

After passing to a tangency-free cluster-adapted model over a
relatively minimal pair $(X_0,\cF_0,\Delta_0)$, write
\[
  K_{\cF_0}=P+N,
  \qquad
  \Delta_0=\sum_i a_iC_i.
\]
The cluster formula takes the form
\[
  \vol(K_{\cF}+\Delta)
  =
  P^2+\sum_i a_i(2-a_i)P\cdot C_i
  +R(\cF_0,\Delta_0)
  +\sum_{p\in\sS_\rho}T_1(p),
\]
where $R(\cF_0,\Delta_0)\geq0$ and all local contributions
$T_1(p)$ are nonnegative; moreover, a regular cluster root contributes
at least $1/4$.

When $\kappa(\cF)=0$, the canonical cyclic cover of degree
$\delta_1(\cF)$ reduces the argument to the numerically trivial case,
giving the lower bound $1/(4\delta_1(\cF))$.
When $\kappa(\cF)=1$, the bigness of $K_{\cF}+\Delta$ and the first pluricanonical pencil yield
a component $C_i$ of $\Delta_0$ such that
\[
  P\cdot C_i\geq\frac1{\delta_2(\cF_0)}.
\]
Together with $a_i(2-a_i)\geq3/4$, this yields the required lower
bound.

Finally, applying these estimates to the tangency-free quotient pair
$(\bar Y,\bar\cG,\bar\Delta)$ and using
\[
  \vol(\cF)
  =
  |G|\,\vol(K_{\bar\cG}+\bar\Delta)
\]
gives Theorem~\ref{thm:main}.

\section{Foliations, tangency, and Zariski decompositions}
\label{sec:preliminaries}

Throughout this section, $X$ is a smooth projective surface over
$\mathbb C$ and $\cF$ is a rank-one foliation on $X$.  Unless otherwise
stated, a boundary divisor will be written as
\[
  \Delta=\sum_{i=1}^{l}a_iC_i,
  \qquad
  a_i\in\left[\frac12,1\right),
\]
where every $C_i$ is irreducible and non-$\cF$-invariant.  We use
$\Delta_{\red}:=\sum_iC_i$ and denote strict transforms by the same
letter with a subscript or a bar.

\subsection{Foliations and reduced singularities}
\label{subsec:foliations-singularities}

A foliation $\cF$ on $X$ is a saturated rank-one subsheaf
$T_{\cF}\subset T_X$. Since $T_{\cF}$ is reflexive, it
determines a Weil divisor \(K_{\cF}\) by
\[
  \cO_X(-K_{\cF})\simeq T_{\cF}.
\]
The singular locus $\Sing(\cF)$ is the finite set where
$T_X/T_{\cF}$ is not locally free.

Locally at a singular point $p$ of $\cF$, the foliation is generated by a
holomorphic vector field
\[
  v=A(x,y)\frac{\partial}{\partial x}
    +B(x,y)\frac{\partial}{\partial y},
\]
where $A$ and $B$ have no common divisorial factor.  Let
$\lambda_1,\lambda_2$ be the eigenvalues of the linear part of $v$ at
$p$.  The singularity is \emph{reduced} if, after possibly interchanging
the eigenvalues, $\lambda_2\neq0$ and
\[
  \frac{\lambda_1}{\lambda_2}\notin\mathbb Q_{>0}.
\]
It is \emph{non-degenerate} if both eigenvalues are nonzero, and is a
\emph{saddle-node} otherwise.  The foliation is reduced if all its
singularities are reduced.  By Seidenberg's theorem, every foliation on
a smooth surface becomes reduced after a sequence of blow-ups
\cite{Sei68}.

Let $\pi\colon(Y,\cG)\to(X,\cF)$ be a proper birational morphism, where
$\cG$ denotes the saturated transform of $\cF$.  Whenever $K_{\cF}$ is
$\mathbb Q$-Cartier, write
\[
  K_{\cG}
  =\pi^*K_{\cF}
   +\sum_{E\subset\Exc(\pi)}a(E,\cF)E.
\]
The foliation $\cF$ is \emph{canonical} if $a(E,\cF)\geq0$ for every prime
divisor $E$ over $X$.

An irreducible curve $C\subset X$ is \emph{$\cF$-invariant} if the foliation is
tangent to $C$ at its general point.  Suppose that $C$ is invariant and
let $p\in C$.  If $f=0$ is a reduced local equation of $C$ and $\omega$
is a local $1$-form defining $\cF$, write
\[
  g\omega=h\,df+f\eta,
\]
where $f$ and $h$ are coprime.  The local $\mathrm Z$-index is
\[
  \mathrm Z(\cF,C,p)
  :=\operatorname{ord}_p
    \left(\left.\frac{h}{g}\right|_C\right),
\]
and
\[
  \mathrm Z(\cF,C)
  :=\sum_{p\in C}\mathrm Z(\cF,C,p).
\]
For an irreducible $\cF$-invariant curve, the index theorem (cf.~\cite[Proposition~2.3]{Bru15}) gives
\begin{equation}
  K_{\cF}\cdot C
  =2p_a(C)-2+\mathrm Z(\cF,C).
  \label{eq:invariant-index-formula}
\end{equation}

At a non-degenerate reduced
singularity, each invariant local separatrix has local
$\mathrm Z$-index one.  At a saddle-node, the strong separatrix has
local $\mathrm Z$-index one, whereas the weak separatrix, when it
exists, has local $\mathrm Z$-index at least two; see
\cite[pp.~30--31]{Bru15}.

We shall also use the following form of the separatrix theorem.

\begin{theorem}[Separatrix theorem]
\label{thm:separatrix}
Let $D$ be a connected compact $\cF$-invariant curve.  Assume that the
singularities of $\cF$ along $D$ are reduced, that $D$ has normal
crossings, and that its intersection matrix is negative definite with
tree dual graph.  Then there is a point $p\in D\cap\Sing(\cF)$ admitting
a separatrix not contained in $D$.
\end{theorem}

\begin{proof}
See \cite[Theorem~3.4]{Bru15}.
\end{proof}

A smooth rational $\cF$-invariant curve $E$ with $E^2=-1$ is called
\emph{$\cF$-exceptional} if, after contracting $E$, the induced foliation is
either regular or has a reduced singularity at the image point.
A foliated surface $(X,\cF)$ is \emph{relatively minimal} if $\cF$ is
reduced and $X$ contains no $\cF$-exceptional curve.  Every foliated
surface admits a relatively minimal model, although it need not be
unique.

A birational self-map $\varphi\colon X\dashrightarrow X$ is said to
\emph{preserve} $\cF$ if the saturated transform $\varphi^*\cF$ coincides
with $\cF$.  We set
\[
  \Aut(X,\cF)
  :=
  \{\varphi\in\Aut(X)\mid \varphi^*\cF=\cF\},
  \qquad
  \Bir(X,\cF)
  :=
  \{\varphi\in\Bir(X)\mid \varphi^*\cF=\cF\}.
\]
When the ambient surface is understood, we simply write
$\Aut(\cF)$ and $\Bir(\cF)$.  If $\cF$ is of general type, then
$\Bir(X,\cF)$ is finite by \cite[Theorem~1]{PS02}.

\subsection{Tangency and blow-ups}
\label{subsec:tangency-blowups}

Let $S\subset X$ be a reduced curve without $\cF$-invariant components.
Around $p\in S$, choose a reduced local equation $f=0$ for $S$ and a
local vector field $v$ generating $\cF$.  Define
\[
  \tang(\cF,S,p)
  :=\dim_{\mathbb C}
    \frac{\cO_{X,p}}{\langle f,v(f)\rangle},
  \qquad
  \tang(\cF,S)
  :=\sum_{p\in S}\tang(\cF,S,p).
\]
The definition is made for the entire reduced curve; in particular, it
also detects singularities of $S$ and intersections between distinct
components.  The tangency formula (cf.~\cite[Proposition~2.2]{Bru15}) is
\begin{equation}
  \tang(\cF,S)=K_{\cF}\cdot S+S^2.
  \label{eq:tangency-formula}
\end{equation}
For an irreducible curve, this is
$\tang(\cF,C)=K_{\cF}\cdot C+C^2$; see
\cite[Proposition~2.2]{Bru15}.

We say that $(\cF,S)$ is \emph{tangency-free} if
$\tang(\cF,S)=0$.  In this case, the irreducible components of $S$ are
smooth and pairwise disjoint, are everywhere transverse to $\cF$, and
are disjoint from $\Sing(\cF)$.

We record the two blow-up formulas used below.  Let
\[
  \sigma\colon(X_1,\cF_1,\Delta_1)
  \longrightarrow(X,\cF,\Delta)
\]
be the blow-up of a regular or reduced singular point $p$, with
exceptional curve $E$, and let $\Delta_1:=\sigma_*^{-1}\Delta$.  Put
\[
  \ell_p:=
  \begin{cases}
    0,&p\notin\Sing(\cF),\\
    1,&p\in\Sing(\cF),
  \end{cases}
  \qquad
  \mu_p:=\mult_p(\Delta)
  =\sum_i a_i\mult_p(C_i).
\]
Then
\begin{equation}
  K_{\cF_1}+\Delta_1
  =\sigma^*(K_{\cF}+\Delta)
   +(1-\ell_p-\mu_p)E.
  \label{eq:adjoint-blowup-formula}
\end{equation}
If $m_p:=\mult_p(\Delta_{\red})$, then
\begin{equation}
  \tang(\cF,\Delta_{\red})
  -\tang(\cF_1,\Delta_{1,\red})
  =m_p(m_p-1+\ell_p).
  \label{eq:tangency-drop}
\end{equation}
Indeed, both identities follow immediately from
$K_{\cF_1}=\sigma^*K_{\cF}+(1-\ell_p)E$ and
$\Delta_{1,\red}=\sigma^*\Delta_{\red}-m_pE$.

\subsection{Zariski decompositions}
\label{subsec:zariski-foliated-chains}

For an $\mathbb R$-divisor $D$ on $X$, set
\[
  \vol(D)
  :=\limsup_{m\to\infty}
    \frac{h^0\bigl(X,\cO_X(\lfloor mD\rfloor)\bigr)}{m^2/2}.
\]
The divisor $D$ is big if and only if $\vol(D)>0$.

Let now $D$ be pseudo-effective.  Its Zariski
decomposition is the unique expression
\[
  D=P(D)+N(D),
\]
where $P(D)$ is nef, $N(D)\geq0$ has negative-definite intersection
matrix when it is nonzero, and
$P(D)\cdot\Gamma=0$ for every component $\Gamma$ of $N(D)$.  Moreover,
\[
  \vol(D)=P(D)^2.
\]
For the two divisors used throughout the paper, we write
\[
  K_{\cF}=P+N,
  \qquad
  K_{\cF}+\Delta=P(\Delta)+N(\Delta).
\]
Thus $P,N$ always refer to the Zariski decomposition of $K_{\cF}$,
whereas $P(\Delta),N(\Delta)$ refer to that of the adjoint divisor.

The Kodaira dimension of $\cF$ is
$\kappa(\cF):=\kappa(K_{\cF})$ when $\cF$ is reduced (or canonical).  When $K_{\cF}$ is pseudo-effective,
its numerical dimension is determined by $P$:
\[
  \nu(\cF)=
  \begin{cases}
    0,&P\equiv0,\\
    1,&P\not\equiv0\text{ and }P^2=0,\\
    2,&P^2>0.
  \end{cases}
\]

We conclude this section with a birational invariance observation.
Although the adjoint divisor itself need not be preserved under an
arbitrary birational modification when the boundary is transformed
strictly, its volume is unchanged between tangency-free models.

\begin{lemma}[Birational invariance of adjoint volume for tangency-free models]
\label{lem:tangency-free-volume}
Let $(X_i,\cF_i,\Delta_i)$, $i=1,2$, be two smooth foliated
triples related by a birational map. Assume that the induced
birational map identifies their foliations and their boundaries by
strict transform, that the foliations $\cF_i$ are reduced, and that
every component of $\Delta_i$ is non-$\cF_i$-invariant with
coefficient in $[0,1]$. If
\[
  \tang(\cF_i,\Delta_{i,\red})=0
  \qquad (i=1,2),
\]
then
\[
  \vol(K_{\cF_1}+\Delta_1)
  =
  \vol(K_{\cF_2}+\Delta_2).
\]
\end{lemma}

\begin{proof}
Take a common resolution
\[
  \mu_i\colon (W,\cH,B)
  \longrightarrow
  (X_i,\cF_i,\Delta_i),
  \qquad i=1,2,
\]
where $B$ is the common strict transform of the boundaries. Factor
each $\mu_i$ into point blow-ups. It is therefore enough to consider
a single blow-up
\[
  \sigma\colon
  (X',\cF',\Delta')
  \longrightarrow
  (X,\cF,\Delta)
\]
of a smooth tangency-free foliated triple, where
$\Delta'=\sigma_*^{-1}\Delta$.

If $p\notin\operatorname{Supp}\Delta$, then $\mu_p=0$ and
\[
  1-\ell_p-\mu_p\in\{0,1\}.
\]
If $p\in\operatorname{Supp}\Delta$, tangency-freeness implies that
$p$ is a regular point of $\cF$ lying on a unique component of
$\Delta$, say of coefficient $a\leq1$. Hence
\[
  \ell_p=0,\qquad
  \mu_p=a,\qquad
  1-\ell_p-\mu_p=1-a\geq0.
\]
Thus, in either case, \eqref{eq:adjoint-blowup-formula} gives
\[
  K_{\cF'}+\Delta'
  =
  \sigma^*(K_{\cF}+\Delta)
  +(1-\ell_p-\mu_p)E
\]
with a nonnegative exceptional coefficient.

Moreover, if $p\notin\operatorname{Supp}\Delta$, then $m_p=0$,
whereas if $p\in\operatorname{Supp}\Delta$, then
$m_p=1$ and $\ell_p=0$. Hence
\[
  m_p(m_p-1+\ell_p)=0,
\]
and \eqref{eq:tangency-drop} shows that
\[
  \tang(\cF',\Delta'_{\red})=0.
\]
Consequently, this argument can be iterated along both morphisms
$\mu_i$, and we obtain
\[
  K_{\cH}+B
  =
  \mu_i^*(K_{\cF_i}+\Delta_i)+\cE_i,
\]
where $\cE_i\geq0$ is $\mu_i$-exceptional. Since adding an effective
exceptional divisor does not change the volume,
\[
  \vol(K_{\cH}+B)
  =
  \vol(K_{\cF_i}+\Delta_i),
  \qquad i=1,2.
\]
The assertion follows.
\end{proof}

In particular, when studying the adjoint volume of a smooth
tangency-free triple, we may replace it by any smooth tangency-free
birational model carrying the corresponding strict-transform
boundary. We shall use this observation by first passing to a
relatively minimal model and then applying clusters until the
boundary becomes tangency-free.

\section{The negative part of \texorpdfstring{$K_{\cF}+\Delta$}{K(F) + Delta}}
\label{sec:negative-part}

This section describes the negative part of $K_{\cF}+\Delta$ and
records the numerical data used later.  Throughout, let $\cF$ be a
reduced foliation on a smooth projective surface $X$, and let
\[
  \Delta=\sum_{i=1}^{l}a_iC_i,
  \qquad
  a_i\in\left[\frac12,1\right),
\]
where every $C_i$ is irreducible and non-$\cF$-invariant.  Assume that
$K_{\cF}+\Delta$ is pseudo-effective, and write
\[
  K_{\cF}+\Delta=P(\Delta)+N(\Delta)
\]
for its Zariski decomposition.

\subsection{Special chains}
\label{subsec:special-chains}

\begin{definition}
Let $\cF$ be a foliation on a surface $X$. A compact curve $\Theta\subset X$ is called an \emph{$\cF$-chain} if 
\begin{enumerate}
\item $\Theta$ is a Hirzebruch--Jung string, $\Theta=\Gamma_1+\cdots+\Gamma_r$;
\item each irreducible component $\Gamma_j$ is $\cF$-invariant;
\item ${\rm Sing}(\cF)\cap \Theta$ are reduced and non-degenerate;
\item ${\rm Z}(\cF,\Gamma_1)=1$ and  ${\rm Z}(\cF,\Gamma_i)=2$ for all $i\geq2$. 
\end{enumerate}
An $\cF$-chain is \emph{maximal} if it cannot be contained in another $\cF$-chain.   (See \cite[Definition 8.1]{Bru15})
\end{definition}

Since every $\Gamma_i$ is rational, the invariant index formula gives
\[
  K_{\cF}\cdot\Gamma_1=-1,
  \qquad
  K_{\cF}\cdot\Gamma_i=0\quad(i\geq2).
\]
Let $M(\Theta)$ be the unique effective $\mathbb Q$-divisor supported
on $\Theta$ satisfying
\[
  M(\Theta)\cdot\Gamma_1=-1,
  \qquad
  M(\Theta)\cdot\Gamma_i=0\quad(i\geq2).
\]

Write $e_i=-\Gamma_i^2$.  For $1\leq i\leq j\leq r$, set
\[
  [e_i,\ldots,e_j]
  :=
  \det\bigl(-\Gamma_k\cdot\Gamma_\ell\bigr)_{i\leq k,\ell\leq j},
\]
and let the empty continuant be one.  Define
\[
  n_{\Theta}:=[e_1,\ldots,e_r],
  \qquad
  \lambda_i:=[e_{i+1},\ldots,e_r]
  \quad(1\leq i\leq r).
\]
Thus $\lambda_r=1$; we shall also
write
\[
  q_{\Theta}:=\lambda_1=[e_2,\ldots,e_r].
\]
The standard Hirzebruch--Jung calculation gives
\begin{equation}
  M(\Theta)
  =\sum_{i=1}^r\gamma_i\Gamma_i,
  \qquad
  \gamma_i=\frac{\lambda_i}{n_{\Theta}}.
  \label{eq:MTheta-coefficients}
\end{equation}
In particular,
\begin{equation}
  -M(\Theta)^2
  =\frac{q_{\Theta}}{n_{\Theta}},
  \qquad
  \gamma_r=\frac1{n_{\Theta}}\leq\frac12.
  \label{eq:MTheta-square}
\end{equation}
We put
\[
  \beta_{\cF}(\Theta)
  :=-M(\Theta)^2
  =\frac{q_{\Theta}}{n_{\Theta}}.
\]

We recall that if $(X,\cF)$ is relatively minimal and $K_{\cF}$ is
pseudo-effective, then the connected components of the negative part
of $K_{\cF}$ are precisely the maximal $\cF$-chains.  More precisely,
if
\[
  K_{\cF}=P+N,
\]
then
\begin{equation}
  N=\sum_{\Xi}M(\Xi),
  \qquad
  \lfloor N\rfloor=0,
  \label{eq:McQuillan-negative-part}
\end{equation}
where $\Xi$ runs over the maximal $\cF$-chains; see \cite[Theorem~8.1]{Bru15} or \cite{McQ08}.

\begin{definition}
An $\cF$-chain $\Theta=\Gamma_1+\cdots+\Gamma_r$ is a
$(\Delta,\cF)$-\emph{chain} if
\[
  \Delta\cdot\Gamma_1<1,
  \qquad
  \Delta\cdot\Gamma_i=0\quad(i\geq2).
\]
We put
\[
  \theta_{\Delta}(\Theta):=\Delta\cdot\Gamma_1.
\]
It is \emph{maximal} if it is not properly contained in another
$(\Delta,\cF)$-chain.
\end{definition}

For a $(\Delta,\cF)$-chain $\Theta$, let $M(\Delta,\Theta)$ be the
unique $\mathbb Q$-divisor supported on $\Theta$ such that
\[
  M(\Delta,\Theta)\cdot\Gamma_i
  =(K_{\cF}+\Delta)\cdot\Gamma_i
  \qquad(1\leq i\leq r).
\]
Then
\begin{equation}
  M(\Delta,\Theta)
  =\bigl(1-\theta_{\Delta}(\Theta)\bigr)M(\Theta),
  \qquad
  0<M(\Delta,\Theta)\leq M(\Theta).
  \label{eq:MDeltaTheta}
\end{equation}
Moreover,
\[
  \theta_{\Delta}(\Theta)\in
  \{0\}\cup\left[\frac12,1\right).
\]

The coefficient restriction has the following useful consequence.  If
$\theta_{\Delta}(\Theta)>0$, then there is a unique component $C_j$ of $\Delta$
meeting $\Theta$; it meets $\Gamma_1$ transversely once, is disjoint
from $\Gamma_i$ for $i\geq2$, and
\[
  \theta_{\Delta}(\Theta)=a_j.
\]
Indeed, two branches, or an intersection of multiplicity at least two,
would give $\Delta\cdot\Gamma_1\geq1$.  Consequently,
\begin{equation}
  M(\Delta,\Theta)\cdot\Delta
  =\theta_{\Delta}(\Theta)\bigl(1-\theta_{\Delta}(\Theta)\bigr)
   \beta_{\cF}(\Theta)=\theta_{\Delta}(\Theta)\bigl(1-\theta_{\Delta}(\Theta)\bigr)\frac{q_{\Theta}}{n_{\Theta}}.
  \label{eq:MDeltaTheta-boundary}
\end{equation}

\subsection{The Zariski decomposition of $K_{\cF}+\Delta$}
\label{subsec:adjoint-zariski-decomposition}

\begin{definition}\label{def:(Delta,F)-exc}
An $\cF$-exceptional curve $E$ is called
$(\Delta,\cF)$-\emph{exceptional} if either
\begin{enumerate}
  \item $K_{\cF}\cdot E=-1$ and $\Delta\cdot E<1$, or
  \item $K_{\cF}\cdot E=\Delta\cdot E=0$.
\end{enumerate}
A $(\Delta,\cF)$-exceptional curve $E$ is \emph{redundant} if there is
a $(K_{\cF}+\Delta)$-non-positive birational morphism
\[
  \sigma\colon(X,\cF,\Delta)\longrightarrow
  (X',\cF',\Delta')
\]
between reduced foliated triples which contracts $E$, and for which
$\sigma(E)$ is either a regular point of $\cF'$ or lies on a
$(\Delta',\cF')$-chain.  The triple $(X,\cF,\Delta)$ is
\emph{irredundant} if $(X,\cF)$ is reduced and contains no redundant
$(\Delta,\cF)$-exceptional curve.
\end{definition}

\begin{remark}[Redundant contractions]
\label{rem:redundant-no-saddle-node}

A redundant $(\Delta,\cF)$-exceptional curve contains no saddle-node
of $\cF$.  Moreover, if
\[
  \sigma\colon(X,\cF,\Delta)
  \longrightarrow(X',\cF',\Delta')
\]
a single blow-up contracting such a curve, then
\[
  \tang(\cF,\Delta_{\red})
  =
  \tang(\cF',\Delta'_{\red}).
\]

Indeed, let $p$ be the image of the contracted curve.  If
$K_{\cF}\cdot E=-1$, then $\ell_p=0$ and
$\mu_p=\Delta\cdot E<1$.  Since the nonzero coefficients of
$\Delta'$ are at least $1/2$, one has $m_p\leq1$.  If
$K_{\cF}\cdot E=\Delta\cdot E=0$, then $\ell_p=1$ and $m_p=0$.
Thus in either case
\[
  m_p(m_p-1+\ell_p)=0,
\]
and the assertion follows from \eqref{eq:tangency-drop}.  Notice that
the coefficient bound $a_i\geq1/2$ is essential here.

If, in addition, $K_{\cF}+\Delta$ is big, then the defining
$(K_{\cF}+\Delta)$-non-positivity of $\sigma$ also gives
\[
  \vol(K_{\cF}+\Delta)
  =
  \vol(K_{\cF'}+\Delta').
\]
\end{remark}

We now turn to the structure of the negative part.

\begin{theorem}[Negative part of $K_{\cF}+\Delta$]
\label{thm:zariski-decomposition-foliated-pairs}
Assume that $(X,\cF,\Delta)$ is irredundant.  If
$\Theta_1,\ldots,\Theta_s$ are all the maximal
$(\Delta,\cF)$-chains, then
\begin{equation}
  N(\Delta)=\sum_{i=1}^sM(\Delta,\Theta_i).
  \label{eq:negative-part-chain-decomposition}
\end{equation}
In particular, $\operatorname{Supp}N(\Delta)$ is the disjoint union of
the maximal $(\Delta,\cF)$-chains.
\end{theorem}

\begin{proof}
Let $\mathfrak{S}$ be the set of non-$\cF$-invariant components of
$N(\Delta)$, and set
\[
  T=\sum_{D\in\mathfrak{S}}\alpha_DD,
  \qquad
  \alpha_D=
  \begin{cases}
    1,&D\not\subseteq\operatorname{Supp}\Delta,\\
    1-a_i,&D=C_i.
  \end{cases}
\]
Thus every $D\in\mathfrak{S}$ has coefficient one in $\Delta+T$.

Let $\Theta_1',\ldots,\Theta_t'$ be all the maximal
$(\Delta+T,\cF)$-chains and put
\[
  V:=\sum_{i=1}^tM(\Delta+T,\Theta_i').
\]
Every $\Theta_i'$ is disjoint from $\operatorname{Supp}T$.  Moreover,
for every component $\Gamma$ of $V$,
\[
  (K_{\cF}+\Delta-V)\cdot\Gamma=0.
\]
The supported negativity lemma therefore gives
\begin{equation}
  0\leq V\leq N(\Delta),
  \qquad
  (N(\Delta)-V)\cdot\Gamma=0
  \quad(\Gamma\subset\operatorname{Supp}V).
  \label{eq:V-below-NDelta}
\end{equation}
Set
\[
  W:=N(\Delta)-V+T.
\]
We prove that $W=0$.

First let $D\in\mathfrak{S}$.  Since $V\cdot D=0$ and
$P(\Delta)\cdot D=0$, we have
\[
W\cdot D\geq (K_{\cF}+\Delta)\cdot D+\alpha_DD^2\geq K_{\cF}\cdot D+D^2
   =\tang(\cF,D)\geq0.
\]
Here, if $D\not\subseteq\operatorname{Supp}\Delta$, we use
$\Delta\cdot D\geq0$ and $\alpha_D=1$; if $D=C_i$, we use
$\Delta\cdot D\geq a_iD^2$ and $\alpha_D=1-a_i$.

Next, if $\Gamma$ is a component of $V$, then
$T\cdot\Gamma=0$, and hence \eqref{eq:V-below-NDelta} gives
\[
  W\cdot\Gamma=0.
\]

Suppose that $W\neq0$.  Its support is contained in
$\operatorname{Supp}N(\Delta)$ and therefore has negative-definite
intersection matrix, which implies $W^2<0$.  We may consequently choose an irreducible
component $C$ of $W$ such that $W\cdot C<0$.  By the preceding two
paragraphs, $C$ is $\cF$-invariant and belongs to neither
$\mathfrak{S}$ nor $\operatorname{Supp}V$.

Let $k$ be the number of chains among the $\Theta_i'$ meeting $C$.
Each such chain is attached to $C$ at its terminal component, and
\eqref{eq:MTheta-square} and \eqref{eq:MDeltaTheta} give
\[
  V\cdot C\leq\frac{k}{2}.
\]
Let $h$ count the remaining singularities on $C$ admitting a
separatrix not contained in the union of $C$ and the chains meeting
$C$.  The separatrix theorem (cf.~Theorem~\ref{thm:separatrix}) gives $h\geq1$.  The invariant index
formula \eqref{eq:invariant-index-formula} then yields
\[
  K_{\cF}\cdot C
  \geq h+k-2+2p_a(C).
\]
It follows that
\begin{align}
  0>W\cdot C
  &=(K_{\cF}+\Delta-V+T)\cdot C\geq h+\frac{k}{2}-2+2p_a(C)+(\Delta+T)\cdot C.
  \label{eq:W-negative-component}
\end{align}
Since every positive coefficient of $\Delta+T$ is at least $1/2$ and
every component of $T$ has coefficient one, inequality
\eqref{eq:W-negative-component} forces
\[
  T\cdot C=0,
  \qquad
  p_a(C)=0,
  \qquad
  h=1,
  \qquad
  k\leq1.
\]
Furthermore,
\[
  V\cdot C-\Delta\cdot C
  >K_{\cF}\cdot C\geq k-1.
\]
Since $K_{\cF}\cdot C$ is an integer, there are precisely two
possibilities:
\begin{enumerate}
  \item $k=0$, $K_{\cF}\cdot C=-1$, and $\Delta\cdot C<1$;
  \item $k=1$, $K_{\cF}\cdot C=0$, and
        $\Delta\cdot C<1/2$, hence $\Delta\cdot C=0$.
\end{enumerate}
In both cases, all singularities of $\cF$ on $C$ are non-degenerate.
If $C^2=-1$, its contraction is $(K_{\cF}+\Delta)$-non-positive and
maps $C$ either to a regular point or to a point on the unique chain
meeting it.  Thus $C$ is redundant, contradicting irredundancy.  If
$C^2\leq-2$, then $C$ itself in the first case, or the union of $C$
with the unique chain meeting it in the second case, extends a
$(\Delta+T,\cF)$-chain.  This contradicts maximality.  Hence $W=0$.

Finally, both $N(\Delta)-V$ and $T$ are effective.  The equality
$W=0$ therefore implies
\[
  T=0,
  \qquad
  N(\Delta)=V.
\]
The maximal $(\Delta+T,\cF)$-chains are now precisely the maximal
$(\Delta,\cF)$-chains, proving the theorem.
\end{proof}

\begin{remark}
A more general version, including a description of the null locus of
the positive part, is given in \cite[Theorem~1.4]{LWX26}.
\end{remark}

The decomposition immediately gives the numerical identities used in
the cluster calculations.

\begin{corollary}\label{cor:NDelta-numerical-data}
With the notation of Theorem~\ref{thm:zariski-decomposition-foliated-pairs},
put $\theta_i:=\theta_{\Delta}(\Theta_i)$.  Then
\begin{align}
  N(\Delta)^2
  &=-\sum_{i=1}^s
    \bigl(1-\theta_i\bigr)^2\beta_{\cF}(\Theta_i),
    \label{eq:NDelta-square}\\
  N(\Delta)\cdot\Delta
  &=\sum_{i=1}^s
    \theta_i\bigl(1-\theta_i\bigr)\beta_{\cF}(\Theta_i).
    \label{eq:NDelta-boundary}
\end{align}
\end{corollary}

\begin{proof}
Distinct maximal chains are disjoint.  The assertions therefore follow
from \eqref{eq:MTheta-square}, \eqref{eq:MDeltaTheta}, and
\eqref{eq:MDeltaTheta-boundary}.
\end{proof}

\begin{corollary}\label{cor:R(F,D)}
Suppose that $(X,\cF,\Delta)$ is irredundant with
$\nu(\cF)\geq0$.  Set
\[
  R(\cF,\Delta)
  :=\sum_j a_j(2-a_j)N\cdot C_j+N^2-N(\Delta)^2.
\]
Then
\begin{equation}
  R(\cF,\Delta)
  \geq
  \bigl(N+N(\Delta)\bigr)\cdot\Delta
  +N^2-N(\Delta)^2
  =P(\Delta)\cdot N
  \geq0.
  \label{eq:R(D,F)>=0}
\end{equation}
\end{corollary}

\begin{proof}
Let $\Theta_1,\ldots,\Theta_s$ be the maximal
$(\Delta,\cF)$-chains and put
\[
  N_{\mathrm{ch}}:=\sum_{i=1}^sM(\Theta_i).
\]
Every $\Theta_i$ is contained in a maximal $\cF$-chain.  Hence
McQuillan's decomposition \eqref{eq:McQuillan-negative-part}, together
with the supported negativity lemma, gives
\[
  0\leq N_{\mathrm{ch}}\leq N.
\]
For each $i$, the observation preceding
\eqref{eq:MDeltaTheta-boundary} gives
\[
  M(\Delta,\Theta_i)\cdot\Delta
  =\sum_j a_j(1-a_j)M(\Theta_i)\cdot C_j.
\]
Summing over the maximal chains, we obtain
\[
N(\Delta)\cdot\Delta
=\sum_j a_j(1-a_j)N_{\mathrm{ch}}\cdot C_j\leq\sum_j a_j(1-a_j)N\cdot C_j.
\]
Indeed, $N-N_{\mathrm{ch}}$ is effective and $\cF$-invariant, whereas
each $C_j$ is non-$\cF$-invariant.  Consequently,
\begin{align*}
  \sum_j a_j(2-a_j)N\cdot C_j
  &=N\cdot\Delta+
    \sum_j a_j(1-a_j)N\cdot C_j\\
  &\geq \bigl(N+N(\Delta)\bigr)\cdot\Delta.
\end{align*}
This proves the first inequality in \eqref{eq:R(D,F)>=0}.

Finally, $N(\Delta)\leq N_{\mathrm{ch}}\leq N$.  Hence
$P\cdot N(\Delta)=0$, while
$P(\Delta)\cdot N(\Delta)=0$.  Using
\[
  \Delta+N-N(\Delta)=P(\Delta)-P,
\]
we compute
\begin{align*}
  &\bigl(N+N(\Delta)\bigr)\cdot\Delta
    +N^2-N(\Delta)^2\\
  &\qquad
   =\bigl(\Delta+N-N(\Delta)\bigr)
     \cdot\bigl(N+N(\Delta)\bigr)\\
  &\qquad
   =\bigl(P(\Delta)-P\bigr)
     \cdot\bigl(N+N(\Delta)\bigr)
   =P(\Delta)\cdot N.
\end{align*}
Since $P(\Delta)$ is nef and $N$ is effective, the last quantity is
nonnegative.
\end{proof}

\section{Finite covers of foliated surfaces}

This section studies finite morphisms between foliated surfaces.  Only
the non-invariant part of the branch locus contributes to the foliated
Riemann--Hurwitz formula; the invariant part is foliated crepant.  When
the horizontal ramification indices are constant over each branch
component, the ramification is encoded by a boundary downstairs.
After resolving the cover in a way adapted to this horizontal branch
divisor, we obtain the smooth models used in the sequel.

\subsection{Riemann--Hurwitz and branch boundaries}
Let
\[
  \Pi\colon (X,\cF)\longrightarrow (Y,\cG)
\]
be a finite surjective morphism of normal foliated surfaces, where
\(\cF\) is the saturated pullback of \(\cG\).  For a prime divisor
\(D\subset X\), let \(e_D\) be the ramification index of \(\Pi\) at
the generic point of \(D\), and put
\[
R_{\Pi}:=
  \sum
  (e_D-1)D,\qquad
  R_{\Pi,v}
  :=
  \sum_{\Pi(D):\cG\text{-invariant}}
  (e_D-1)D,\qquad R_{\Pi,h}=R_{\Pi}-R_{\Pi,v}.
\]

\begin{proposition}[Foliated Riemann--Hurwitz]
\label{prop:foliated-Riemann-Hurwitz}
Assume that \(K_{\cG}\) is \(\mathbb Q\)-Cartier.  Then
\[
  K_{\cF}=\Pi^*K_{\cG}+R_{\Pi,h}.
\]
\end{proposition}

\begin{proof}
The usual Riemann--Hurwitz formula and the identities
\[
  K_X=\Pi^*K_Y+R_\Pi,
  \qquad
  K_Y=K_{\cG}+N_{\cG}^*,
  \qquad
  K_X=K_{\cF}+N_{\cF}^*
\]
show that it suffices to prove
\[
  N_{\cF}^*
  =
  \Pi^*N_{\cG}^*+R_{\Pi,v}.
\]

Since this is an equality of rank-one reflexive sheaves, it is enough
to verify it in codimension one. Let $D\subset X$ be a ramification
divisor, let $p\in D$ be a general point, and set $q=\Pi(p)$. Denote
by $e$ the ramification index of $\Pi$ along $D$.

Suppose first that $\Pi(D)$ is non-$\cG$-invariant. There are local
coordinates $(x,y)$ around $q$ and $(u,v)$ around $p$ such that
\[
  \cG\colon \mathrm{d}y=0,
  \qquad
  \Pi(u,v)=(u^e,v).
\]
The conormal sheaves $N_{\cG}^*$ and $N_{\cF}^*$ are generated by
$\mathrm{d}y$ and $\mathrm{d}v$, respectively. Since
\[
  \Pi^*(\mathrm{d}y)=\mathrm{d}v,
\]
there is no correction along $D$.

Suppose next that $\Pi(D)$ is $\cG$-invariant. We may choose local
coordinates such that
\[
  \cG\colon \mathrm{d}y=0,
  \qquad
  \Pi(u,v)=(u,v^e).
\]
In this case,
\[
  \Pi^*(\mathrm{d}y)=e v^{e-1}\mathrm{d}v,
\]
whereas the saturated conormal sheaf $N_{\cF}^*$ is generated by
$\mathrm{d}v$. Hence
\[
  N_{\cF}^*
  =
  \Pi^*N_{\cG}^*+(e-1)D
\]
along $D$.

Therefore $N_{\cF}^*=\Pi^*N_{\cG}^*+R_{\Pi,v}$, and the assertion follows.
\end{proof}

Write the reduced branch divisor of \(\Pi\) as
\[
  B=B_h+B_v,
\]
where $B_h$ and $B_v$ are the unions of the non-\(\cG\)-invariant
and $\cG$-invariant components, respectively.  

If $\Pi$ is Galois, then 
\[
R_{\Pi}=\Pi^*\Delta_{\Pi},\qquad R_{\Pi,h}=\Pi^*\Delta_{\Pi,h},
\]
where 
\[
\Delta_{\Pi}
  :=
  \sum_{B_i\subset B}
  \left(1-\frac1{n_i}\right)B_i,\qquad 
  \Delta_{\Pi,h}
  :=
  \sum_{B_i\subseteq B_h}
  \left(1-\frac1{n_i}\right)B_i.
\]
Proposition~\ref{prop:foliated-Riemann-Hurwitz} gives
\begin{equation}
\label{eq:quotient-canonical-pullback}
  K_{\cF}=\Pi^*(K_{\cG}+\Delta_{\Pi,h}).
\end{equation}

Finally, the volume of a divisor scales by the degree of a finite
map.  Hence \eqref{eq:quotient-canonical-pullback} yields
\begin{equation}
\label{eq:finite-cover-volume-scaling}
  \vol(K_{\cF})
  =
  (\deg\Pi)\,
  \vol(K_{\cG}+\Delta_{\Pi,h}).
\end{equation}
In particular, if \(\Pi\) is the quotient by a finite group \(G\),
then \(\deg\Pi=|G|\).  Thus the quotient problem is reduced to a
lower bound for the volume of the branch pair downstairs.

\subsection{Tangency-free pullbacks}

We next arrange that the branch pair downstairs is sufficiently
regular.  The following local statement is the point at which the
tangency-free hypothesis enters.

\begin{proposition}[Canonicality of tangency-free pullbacks]
\label{prop:finite-cover-canonical-pullback}
Let $Y$ be a smooth surface, let $X$ be a normal surface, and let
\[
  \Pi\colon(X,\cF)\longrightarrow(Y,\cG)
\]
be a finite surjective morphism.  Assume that $\cG$ is reduced and
that $\cF$ is its saturated pullback.  Write the reduced branch
divisor as $B=B_h+B_v$, according as its components are
non-$\cG$-invariant or $\cG$-invariant.  If
\[
  \tang(\cG,B_h)=0,
\]
then $X$ has at worst Hirzebruch--Jung singularities and $\cF$
has canonical singularities.

More precisely, let $\rho\colon\widetilde X\to X$ be the minimal
resolution, and let $\mathscr P_{\mathrm{mix}}$ be the set of
singular points of $X$ lying over $B_h\cap B_v$.  For
$p\in\mathscr P_{\mathrm{mix}}$, order the exceptional string from
the horizontal branch to the invariant branch and denote it by
\[
  \Theta_p=E_1+\cdots+E_r.
\]
Then
\begin{equation}
\label{eq:mixed-cover-discrepancy}
  K_{\widetilde\cF}
  =
  \rho^*K_{\cF}
  +
  \sum_{p\in\mathscr P_{\mathrm{mix}}}M(\Theta_p),
\end{equation}
In particular, every discrepancy is nonnegative.
\end{proposition}

\begin{proof}
We work locally in the complex analytic category.  The condition
$\tang(\mathcal G,B_h)=0$
implies that $\mathcal G$ is regular and transverse along $B_h$, and
that the components of $B_h$ are smooth and pairwise disjoint. 
Together with the local structure of the invariant
branches of a reduced foliation, this shows that $B_h+B_v$ is a
normal crossing divisor.

Fix $p\in X$, and set $q=\Pi(p)$.  
 If $q\notin B$, then $\Pi$ is étale
at $p$, and hence $\mathcal F$ is reduced at $p$.  We may therefore
assume that $q\in B$.
Choose local coordinates $(x,y)$ at $q$ such that the components of
$B$ through $q$ are among $(x=0)$ and $(y=0)$.  Choose sufficiently
divisible positive integers $r$ and $s$ such that the Kummer cover
\[
  \tau\colon U\longrightarrow Y,
  \qquad
  x=u^r,\quad y=v^s,
\]
dominates $\Pi$.  If only one boundary component passes through $q$,
we take the exponent corresponding to the other coordinate to be
$1$.  Thus, after choosing the component lying above $p$, we obtain a
factorization
\[
  U\xrightarrow{\theta}X\xrightarrow{\Pi}Y,
  \qquad
  \tau=\Pi\circ\theta.
\]
The cover $\tau$ is Galois with group
$\Gamma=\mu_r\times\mu_s$.
Since it factors through the normal surface $X$, locally
$X\simeq U/H$
for some subgroup $H\subseteq\Gamma$.
By \cite[Theorem~III.5.2]{BHPV04}, the germ $(X,p)$ is either smooth or
a Hirzebruch--Jung singularity.

There are three local configurations.

\smallskip
\noindent
\emph{Non-invariant branch.} Since the components of $B_h$ are
pairwise disjoint, there is only one such component.  
We may write
\[
  \cG\colon\mathrm dy=0,
  \qquad
  B_h=(x=0).
\]
Here $s=1$, and the saturated pullback foliation on $U$ is defined by
$\mathrm{d}v=0$.  It descends to a regular foliation on $X$, since the
action of $H\subseteq\mu_r\times\{1\}$ is trivial in the $v$-direction.
Thus $\mathcal F$ is regular at $p$.

\smallskip
\noindent
\emph{Invariant branch.}
The pullback of the reduced local normal form of \(\cG\) by the
corresponding Kummer map is again reduced. 
If $q$ lies on a unique component of $B_v$, then
$H\subseteq\{1\}\times\mu_s$ acts only on the $v$-coordinate.
Hence $X\simeq U/H$ is smooth and $\cF$ is reduced at $p$.
It remains to consider a crossing of two invariant branches.
The marked minimal resolution has the form
\[
  S_x-E_1-\cdots-E_r-S_y.
\]
Every \(E_j\) is \(\widetilde\cF\)-invariant, and the only
singularities on it are the two reduced corner singularities.  Hence
\[
  \mathrm Z(\widetilde\cF,E_j)=2,
  \qquad
  K_{\widetilde\cF}\cdot E_j=0.
\]
Negative definiteness of the exceptional string then gives
\[
  K_{\widetilde\cF}=\rho^*K_{\cF}
\]
over such a point.

\smallskip
\noindent
\emph{Mixed branch.}
We may choose coordinates so that
\[
  \cG\colon\mathrm dy=0,
  \qquad
  B_h=(x=0),
  \qquad
  B_v=(y=0).
\]
The Kummer pullback is regular and is defined by $\mathrm dv=0$.
Write the marked resolution as
\[
  T-E_1-\cdots-E_r-S,
\]
where $T$ and $S$ are the strict transforms of the non-invariant
and invariant branches.  The function $y\circ\Pi$ is a local first
integral of $\cF$, and
\[
  \operatorname{div}\bigl(\rho^*(y\circ\Pi)\bigr)
  =c_0S+\sum_{j=1}^r c_jE_j,
  \qquad
  c_0,c_j>0.
\]
At a corner of this divisor, it has the form
\[
  x_1^\alpha y_1^\beta\varepsilon,
  \qquad
  \alpha,\beta>0,
  \qquad
  \varepsilon(0,0)\neq0.
\]
The saturated differential has linear part
\[
  \alpha y_1\,\mathrm dx_1+
  \beta x_1\,\mathrm dy_1,
\]
and hence defines a reduced non-degenerate singularity.  Ordering the
string from $T$ to $S$, we obtain
\[
  \mathrm Z(\widetilde\cF,E_1)=1,
  \qquad
  \mathrm Z(\widetilde\cF,E_j)=2
  \quad(j\geq2).
\]
Therefore \(\Theta=E_1+\cdots+E_r\) is a
\(\widetilde\cF\)-chain.  If
\[
  K_{\widetilde\cF}=\rho^*K_{\cF}+A,
\]
then
\[
  A\cdot E_1=-1,
  \qquad
  A\cdot E_j=0\quad(j\geq2).
\]
By the defining property of \(M(\Theta)\) and the negative
definiteness of the string, \(A=M(\Theta)\).  This proves
\eqref{eq:mixed-cover-discrepancy} and completes the proof.
\end{proof}

\begin{remark}
In the notation of the proof, the pullback foliation $\mathcal H$ on
the Kummer cover $U$ is reduced, hence canonical, and
\[
  (X,\mathcal F)\simeq (U,\mathcal H)/H.
\]
The canonicality of $\mathcal F$ therefore also follows from the
general result of McQuillan and Panazzolo that a finite quotient of a
canonical foliation is canonical; see
\cite[Corollary~III.i.5]{MP13}.

The condition $\tang(\cG,B_h)=0$ is stronger than the sharp local
condition for a fixed ramification index; see Example~\ref{exa:tang=0-notsharp}, but it is uniform in the
cover and is exactly the condition needed below.
\end{remark}

\begin{example}\label{exa:tang=0-notsharp}
Let $p$ be a regular point of $\cG$ lying on
a smooth non-$\cG$-invariant branch component, and put
$t=\mathrm{tang}(\cG,B_h,p)$.
In suitable local coordinates, the foliation, the branch component,
and a cyclic cover of local ramification index $e$ are given by
\[
  \cG\colon\mathrm{d}x=0,\qquad
  B_h=(x+y^{t+1}=0),\qquad
  x=u^e-v^{t+1},\quad y=v.
\]
The pullback foliation is defined by
\[
  e u^{e-1}\,\mathrm{d}u
  -(t+1)v^t\,\mathrm{d}v.
\]
It is regular when $t=0$.  If $t=1$ and $e=2$, its linear part is
\[
  2u\,\mathrm{d}u-2v\,\mathrm{d}v,
\]
which defines a reduced non-degenerate singularity with eigenvalue
quotient $-1$.  In all the remaining cases with $t>0$, the pullback
singularity is not reduced.  Thus the sharp local condition is
\[
  t\leq1\quad\text{if }e=2,
  \qquad
  t=0\quad\text{if }e\geq3.
\]
If $p\in\operatorname{Sing}(\cG)\cap B_h$, the corresponding local
calculation shows that the pullback singularity is not reduced.
\end{example}

\subsection{A tangency-free model of a finite cover}

\begin{proposition}[Tangency-free model of a finite cover]
\label{prop:tangency-free-quotient-model}
Let
\[
  \Pi\colon(X,\cF)\longrightarrow(Y,\cG)
\]
be a finite surjective morphism of normal foliated projective
surfaces, with \(\cF\) the saturated pullback of \(\cG\).  There is
a commutative diagram
\[
\begin{tikzcd}[column sep=large,row sep=large]
(\widetilde X,\widetilde\cF)
  \arrow[r,"\rho"]
  \arrow[dr,"\widetilde\Pi"']
&
(X',\cF')
  \arrow[r,"\nu"]
  \arrow[d,"\Pi'"]
&
(X,\cF)
  \arrow[d,"\Pi"]
\\
&
(Y',\cG')
  \arrow[r,"\sigma"']
&
(Y,\cG),
\end{tikzcd}
\]
with the following properties:
\begin{enumerate}
\item $Y'$ and $\widetilde X$ are smooth, while $\cG'$ and
      $\widetilde\cF$ are reduced;
\item $X'$ is the normalization of 
      $X\times_Y Y'$, the surface $X'$ has at worst
      Hirzebruch--Jung singularities, and $\cF'$ is canonical;
\item $\rho$ is the minimal resolution of $X'$;
\item if $B_h'$ is the non-invariant part of the reduced branch
      divisor of $\Pi'$, then
      $\tang(\cG',B_h')=0$.
\end{enumerate}
Write \(B'=B_h'+B_v'\) for the reduced branch divisor of \(\Pi'\),
and let \(\mathscr P'_{\mathrm{mix}}\) be the set of singular points
of \(X'\) lying over \(B_h'\cap B_v'\).  Then
\begin{align}
  K_{\cF'}
  &=
  (\Pi')^*K_{\cG'}+R_{\Pi',h},
  \label{eq:finite-cover-package-RH}
  \\
  K_{\widetilde\cF}
  &=
  \rho^*K_{\cF'}+E,
  \qquad E=
  \sum_{p\in\mathscr P'_{\mathrm{mix}}}M(\Theta_p)\geq0.
  \label{eq:quotient-package-discrepancy}
\end{align}
Moreover,
\begin{equation}
\label{eq:finite-cover-package-volume}
  \vol(K_{\widetilde\cF})=\vol(K_{\cF'}).
\end{equation}
\end{proposition}

\begin{proof}
Choose a resolution
\[
  \sigma_1\colon(Y_1,\cG_1)\longrightarrow(Y,\cG)
\]
such that \(Y_1\) is smooth and \(\cG_1\) is reduced, and let
\(X_1\) be the normalization of 
\(X\times_Y Y_1\).  Write
\[
  B_1=B_{1,h}+B_{1,v}
\]
for the reduced branch divisor of \(X_1\to Y_1\).

By successively blowing up the points of positive tangency of $B_{1,h}$ with the transformed foliation,
we obtain
\[
  \sigma_2\colon(Y',\cG')\longrightarrow(Y_1,\cG_1)
\]
such that
\[
  \tang(\cG',B_h')=0.
\]
Every \(\sigma_2\)-exceptional curve is \(\cG'\)-invariant.  Hence
the horizontal branch divisor of the normalized base change
\(\Pi'\colon X'\to Y'\) is precisely the strict transform \(B_h'\).
Proposition~\ref{prop:finite-cover-canonical-pullback} now shows that
\(X'\) has at worst Hirzebruch--Jung singularities and that \(\cF'\)
is canonical.  Taking the minimal resolution \(\rho\) gives the
diagram and properties (1)--(4).

Equation~\eqref{eq:finite-cover-package-RH} is
Proposition~\ref{prop:foliated-Riemann-Hurwitz}, and
\eqref{eq:quotient-package-discrepancy} is the local discrepancy
formula \eqref{eq:mixed-cover-discrepancy}.  Since adding an effective
exceptional divisor does not change volume, we have
$\vol(K_{\widetilde\cF})=\vol(K_{\cF'})$,
which proves \eqref{eq:finite-cover-package-volume}.
\end{proof}

\begin{corollary}[Galois covers]
\label{cor:Galois-cover}

In the notation of
Proposition~\ref{prop:tangency-free-quotient-model}, assume that
$\Pi$ is Galois with Galois group $G$. Let $B_i'$ be the irreducible
components of $B_h'$, and let $n_i$ be the ramification index over
$B_i'$. Set
\[
  \Delta_{\Pi',h}
  :=
  \sum_{B_i'\subseteq B_h'}
  \left(1-\frac1{n_i}\right)B_i'.
\]
Then
\[
  K_{\cF'}
  =
  (\Pi')^*(K_{\cG'}+\Delta_{\Pi',h})
\]
and
\[
  \vol(K_{\widetilde\cF})
  =
  |G|\,\vol(K_{\cG'}+\Delta_{\Pi',h}).
\]
\end{corollary}

\begin{proof}
The normalized base change $\Pi'$ is again Galois with Galois group
$G$. Hence all ramification divisors over $B_i'$ have the same
ramification index $n_i$, and therefore
\[
  R_{\Pi',h}=(\Pi')^*\Delta_{\Pi',h}.
\]
The first equality follows from the foliated Riemann--Hurwitz formula.
The second follows from \eqref{eq:finite-cover-package-volume}
and the finite-degree scaling of volume.
\end{proof}

We conclude this section by relating the preceding construction to the birational
quotient used in the main theorem. 

Let $\cF$ be a canonical foliation on a smooth projective surface
$X$, and let
\[
  G\subseteq\operatorname{Bir}(X,\cF)
\]
be a finite subgroup, and let
\[
  q\colon(X,\cF)\dashrightarrow
  (X,\cF)/G=(Y,\cG)
\]
be the rational quotient map.

Resolve the indeterminacy of $q$ by a sequence of blow-ups
$\mu\colon(\widehat X,\widehat\cF)\to(X,\cF)$.
The resulting morphism
$\widehat q:=q\circ\mu
  \colon(\widehat X,\widehat\cF)\to(Y,\cG)$
admits a Stein factorization
\[
  (\widehat X,\widehat\cF)
  \xrightarrow{\eta}
  (X_0,\cF_0)
  \xrightarrow{q_0}
  (Y,\cG).
\]
Thus we have a commutative diagram
\[
\begin{tikzcd}[column sep=large,row sep=large]
(\widehat X,\widehat\cF)
  \arrow[r,"\mu"]
  \arrow[d,"\eta"']
&
(X,\cF)
  \arrow[d,dashed,"q"]
\\
(X_0,\cF_0)
  \arrow[r,"q_0"']
&
(X,\cF)/G=(Y,\cG),
\end{tikzcd}
\]
where $\eta$ is birational and $q_0$ is a finite Galois morphism with
Galois group $G$. In particular, $\deg q_0=\deg q=|G|$.

Applying Corollary~\ref{cor:Galois-cover} to $q_0$, we obtain a smooth
tangency-free quotient pair $(\bar Y,\bar\cG,\bar\Delta)$ and a smooth
foliated surface $(\widetilde X,\widetilde\cF)$ birational to
$(X,\cF)$ such that
\[
  \bar\Delta
  =
  \sum_i\left(1-\frac1{n_i}\right)B_i,
  \qquad
  \tang(\bar\cG,\bar\Delta_{\red})=0.
\]
By the birational invariance of the canonical volume,
\[
  \vol(K_{\cF})=\vol(K_{\widetilde\cF})
  =
  |G|\,\vol(K_{\bar\cG}+\bar\Delta).
\]

In the next section, we pass from the tangency-free quotient pair
$(\bar Y,\bar\cG,\bar\Delta)$ to a relatively minimal model and
recover its adjoint volume through local cluster corrections.

\section{The cluster formula for adjoint volumes}
\label{sec:cluster-formula}

This section establishes the birational decomposition of adjoint
volume used in the proof of
Theorem~\ref{thm:adjoint-volume-bounds}.

Let $(X,\cF,\Delta)$ be a smooth tangency-free foliated triple such
that $K_{\cF}+\Delta$ is big.  Choose a relatively minimal reduction
\[
  \pi\colon(X,\cF)\longrightarrow(X_0,\cF_0)
\]
and set $\Delta_0:=\pi_*\Delta$.  Since every $\pi$-exceptional curve
is $\cF$-invariant, while $\Delta$ has no $\cF$-invariant components,
no component of $\Delta$ is contracted.

Applying the clusters described below to
$(X_0,\cF_0,\Delta_0)$, we obtain a cluster-adapted morphism
\[
  \rho'\colon
  (X',\cF',\Delta')
  \longrightarrow
  (X_0,\cF_0,\Delta_0),
  \qquad
  \Delta'=(\rho')_*^{-1}\Delta_0,
\]
such that
\[
  \tang(\cF',\Delta'_{\red})=0.
\]
The original and the resulting triples are smooth tangency-free
birational models carrying the same strict-transform boundary.
Hence Lemma~\ref{lem:tangency-free-volume} gives
\[
  \vol(K_{\cF}+\Delta)
  =
  \vol(K_{\cF'}+\Delta').
\]

Replacing the original triple by $(X',\cF',\Delta')$ and dropping
the primes, we henceforth write
\[
  \rho\colon
  (X,\cF,\Delta)
  \longrightarrow
  (X_0,\cF_0,\Delta_0),
  \qquad
  \Delta=\rho_*^{-1}\Delta_0,
\]
for a cluster-adapted resolution with tangency-free output.  Since
\[
  \rho_*(K_{\cF}+\Delta)=K_{\cF_0}+\Delta_0,
\]
the divisor $K_{\cF_0}+\Delta_0$ is big.

Throughout this section, we assume
\[
  \nu(\cF_0)\geq0,
\]
or equivalently that $K_{\cF_0}$ is pseudo-effective.  The initial
triple is irredundant because $(X_0,\cF_0)$ is relatively minimal,
and Proposition~\ref{prop:cluster-preserves-irredundancy} shows that
irredundancy is preserved after every complete cluster.

We shall prove
\[
  \vol(K_{\cF}+\Delta)
  =
  \sV(\cF_0,\Delta_0)
  +
  \sum_{p\in\sS_\rho}T_1(p).
\]
For statements concerning a single cluster, we suppress the stage
index and write $(X,\cF,\Delta)$ for its input and
$(X',\cF',\Delta')$ for its output.  All local quantities are computed
on the input triple.

\subsection{Clusters and cluster-adapted resolutions}
\label{subsec:cluster-adapted-resolutions}
We first record the effect of one blow-up and then group the successive
infinitely near centers into clusters.
Let 
\[
\sigma\colon(X_1,\cF_1,\Delta_1,E)\to(X,\cF,\Delta,p)
\]
be the blow-up of a point $p\in\Supp\Delta$, with exceptional divisor $E$, where 
\[
\Delta_1:=\sigma_*^{-1}\Delta=\sum a_i\sigma_*^{-1}C_i.
\]
Set 
\[
\ell_p:=\ell(p)\in\{0,1\},\qquad m_p:=\mult_p(\Delta_{\red}),\qquad \mu_p:=\mult_p(\Delta)=\sum a_i\mult_p(C_i).
\]
Then 
\begin{equation}\label{eq:one-blowup-adjoint}
K_{\cF_1}+\Delta_1=\sigma^*(K_{\cF}+\Delta)+(1-\ell_p-\mu_p)E.
\end{equation}
\begin{definition}[Cluster]
A \emph{cluster over $p$} is a minimal sequence of blow-ups
\[
\varrho_p:(X',\cF',\Delta')
\longrightarrow
(X,\cF,\Delta),\qquad\Delta'=(\varrho_p)_*^{-1}\Delta
\]
with centers $q_1,\dots,q_{r_p}$, satisfying the following conditions:
\begin{enumerate}
\item 
$q_1=p\in\Supp\Delta$ and for $i\ge2$
\[
q_i\in\Supp\Delta_{i-1}\cap\Sing(\cF_{i-1})\cap\varrho_{p,i-1}^{-1}(p)
\]
where
$\varrho_{p,i}:=\sigma_1\circ\cdots\circ\sigma_i$ and
$\varrho_{p,0}:={\rm id}_X$.
Here $(X_i,\cF_i,\Delta_i)$ denotes the transformed triple after the first $i$ blow-ups;
\item 
\[
\Supp\Delta'\cap\varrho_p^{-1}(p)\cap\Sing(\cF')=\emptyset.
\]
\end{enumerate}
The sequence is required to be minimal with these properties.
The point $p=q_1$ is called the \emph{root} of the cluster.
\end{definition}
Thus every center of a cluster except its root satisfies 
\begin{equation}\label{eq:cluster-later-centers-singular}
\ell(q_i)=1,\qquad (i\ge2).
\end{equation}
This is a numerical reason for grouping the infinitely near centers into a single cluster.
\begin{definition}[Cluster-adapted resolution]
A \emph{cluster-adapted resolution} of
$(X_0,\cF_0,\Delta_0)$ is a composition
\[
  \rho\colon
  (X_r,\cF_r,\Delta_r)
  \longrightarrow
  (X_0,\cF_0,\Delta_0),
  \qquad
  \Delta_r=\rho_*^{-1}\Delta_0,
\]
obtained by applying clusters successively to the current transformed
triple until
\[
  \tang(\cF_r,\Delta_{r,\red})=0.
\]
\end{definition}

After a cluster is completed, we choose a remaining point of positive
tangency, if any, as the root of the next cluster. Since
$\tang(\cF_i,\Delta_{i,\red})$ is a nonnegative integer and every complete
cluster strictly decreases it, the procedure terminates after finitely
many clusters.
We write such a resolution as
\[
  (X_r,\cF_r,\Delta_r)
  \xrightarrow{\varrho_{p_r}}
  \cdots
  \xrightarrow{\varrho_{p_2}}
  (X_1,\cF_1,\Delta_1)
  \xrightarrow{\varrho_{p_1}}
  (X_0,\cF_0,\Delta_0).
\]
In the global notation fixed at the beginning of the section, we identify
$(X_r,\cF_r,\Delta_r)$ with $(X,\cF,\Delta)$.
Here $p_i\in X_{i-1}$ is the root of the corresponding cluster.  We regard
\[
  \sS_\rho:=\{p_1,\ldots,p_r\}
\]
as an indexed set: its elements lie on different intermediate models.  Every
local invariant attached to $p_i$ is computed on
$(X_{i-1},\cF_{i-1},\Delta_{i-1})$.

For $p=p_i\in\sS_\rho$, put
\[
t_p:=\tang(\cF_{i-1},\Delta_{i-1,\red},p),
\qquad
\sS_{\rho,\ell,m}:=\{p\in\sS_{\rho}\mid \ell_p=\ell,\,m_p=m\}.
\]
The roots are divided into
\[
\sS_{\rho,0,1},\qquad\sS_{\rho,0,m}\quad(m\ge2),\qquad\sS_{\rho,1,m}\quad(m\ge1),
\]
which will be called respectively \emph{simple regular},
\emph{multiple regular}, and \emph{singular roots}.
For a simple regular root, the additional condition $t_p>0$ is part of the definition.
For a multiple regular or a singular root, positive local tangency is automatic.
We may and will first perform all clusters rooted at singular points of
the current foliation. Once these clusters have been completed,
the transformed boundary is disjoint from the singular locus.
Consequently, every subsequent cluster has a regular root, that is,
\(\ell_p=0\).

\begin{proposition}[Preservation of irredundancy]
\label{prop:cluster-preserves-irredundancy}
Let
\[
\varrho_p\colon(X',\cF',\Delta')
\longrightarrow(X,\cF,\Delta)
\]
be a complete cluster starting from an irredundant triple such that
\(K_{\cF}\) is pseudo-effective. Then
\((X',\cF',\Delta')\) is irredundant.
\end{proposition}
\begin{proof}
Let \(E\) be an irreducible component of
\(\varrho_p^{-1}(p)\). If a later blow-up center lies on the successive
strict transform of \(E\), then \(E^2\leq-2\) on \(X'\), so \(E\) is
not \(\cF'\)-exceptional.
Suppose that \(E\) is not blown up again, and let \(q\) be the center
whose blow-up creates \(E\). If \(q\) is singular for the transformed
foliation, then
\[
K_{\cF'}\cdot E=0,
\qquad
\Delta'\cdot E=\mu_q>0,
\]
so \(E\) is not \((\Delta',\cF')\)-exceptional.
It remains to consider the case where \(q\) is regular. Since every
center except the root is singular, one has \(q=p\). If \(p\) is
multiple regular, then
\[
K_{\cF'}\cdot E=-1,
\qquad
\Delta'\cdot E=\mu_p
\geq\frac12m_p\geq1,
\]
and again \(E\) is not \((\Delta',\cF')\)-exceptional. If \(p\) is
simple regular, positive tangency forces the strict transform of
\(\Delta_{\mathrm{red}}\) to pass through the singular point of the
transformed foliation on the first exceptional divisor. Since the
cluster is complete, this point is blown up later, and hence that
exceptional divisor has self-intersection at most \(-2\).
Thus no irreducible component of \(\varrho_p^{-1}(p)\) is redundant.
Suppose now that \(C'\subset X'\) is a redundant
\((\Delta',\cF')\)-exceptional curve not contained in
\(\varrho_p^{-1}(p)\), and set
\[
C:=(\varrho_p)_*C'.
\]
If no center lies on the successive strict transforms of \(C\), then
\(\varrho_p\) is an isomorphism near \(C'\), and \(C\) was already
redundant on \((X,\cF,\Delta)\), contradicting irredundancy.
Suppose therefore that some center lies on the successive strict
transforms of \(C\). The curve \(C\) cannot be the weak separatrix of a
saddle-node: otherwise its successive strict transform would remain
the weak separatrix of a saddle-node, and \(C'\) would contain a
saddle-node, contradicting its redundancy.
The Camacho--Sad formula and the local computations in
\cite[pp.~30--31]{Bru15} give
\[
C^2\leq0.
\]
Since \((C')^2=-1\) and at least one center lies on the successive strict
transforms of \(C\),
\[
C^2\geq(C')^2+1=0.
\]
Thus \(C^2=0\). Equality also shows that precisely one center, with
multiplicity one, lies on the successive strict transforms of \(C\).
Since \(C'\) is smooth rational, so is \(C\).
If \(C\cap\Sing(\cF)\neq\varnothing\), the same local computation shows
that \(C\) contains a unique saddle-node and is its strong separatrix.
Hence
\[
\mathrm Z(\cF,C)\in\{0,1\},
\]
according as \(C\cap\Sing(\cF)\) is empty or nonempty. Therefore
\[
K_{\cF}\cdot C
=
2p_a(C)-2+\mathrm Z(\cF,C)
\leq-1.
\]
On the other hand, \(C\) is nef because it is irreducible and \(C^2=0\).
This contradicts the pseudo-effectivity of \(K_{\cF}\).
Hence no redundant \((\Delta',\cF')\)-exceptional curve occurs on
\(X'\), and \((X',\cF',\Delta')\) is irredundant.
\end{proof}

\begin{remark}[The role of pseudo-effectivity]
\label{rem:role-of-pseudo-effectivity}

For every complete cluster, the blow-up formula gives
\[
  K_{\cF'}
  =
  \varrho_p^*K_{\cF}
  +(1-\ell_p)\cE_1,
\]
where $\cE_1$ is the total transform of the exceptional curve created
by the first blow-up.  Hence $K_{\cF'}$ is pseudo-effective whenever
$K_{\cF}$ is pseudo-effective.

This hypothesis is used in the preceding proof only to exclude a
smooth rational $\cF$-invariant curve $C$ satisfying
\[
  C^2=0,
  \qquad
  K_{\cF}\cdot C<0.
\]
Indeed, such a curve is nef, contradicting the pseudo-effectivity of
$K_{\cF}$.

If $K_{\cF}$ is not pseudo-effective, then, by Miyaoka's theorem,
$\cF$ is induced by a rational fibration; see
\cite[Theorem~7.1]{Bru15}.  In this case the above
obstruction is a smooth rational fiber $C$ with
\[
  C^2=0,
  \qquad
  C\cap\Sing(\cF)=\varnothing,
  \qquad
  K_{\cF}\cdot C=-2.
\]
Blowing up the relevant point of $C$ makes its strict transform a
redundant $(-1)$-curve; contracting it is precisely the elementary
transformation required in the rational-fibration case.
\end{remark}

\begin{corollary}
\label{cor:cluster-intermediate-models-irredundant}
Let
\[
\rho\colon
(X_r,\cF_r,\Delta_r)
\longrightarrow
(X_0,\cF_0,\Delta_0)
\]
be a cluster-adapted resolution starting from an irredundant triple
(for instance, one for which $(X_0,\cF_0)$ is relatively minimal) with
$\nu(\cF_0)\geq0$. Then every triple obtained after the completion of a
cluster is irredundant and has pseudo-effective foliated canonical
divisor. Consequently, the negative-part theorem and
Corollary~\ref{cor:R(F,D)} apply at the beginning and the end of every
cluster.
\end{corollary}
\begin{proof}
Apply Proposition~\ref{prop:cluster-preserves-irredundancy}
successively to the complete clusters of $\rho$. The
pseudo-effectivity of the foliated canonical divisor at every stage
follows from Remark~\ref{rem:role-of-pseudo-effectivity}.
\end{proof}

\subsection{The assembled volume and its cluster variation}
\label{subsec:assembled-volume}
We introduce a weighted tangency term that isolates the boundary
contribution to the variation of the adjoint self-intersection along
a cluster.

\begin{definition}
\label{def:weighted-tangency}
Set
\[
  \tau(\cF,\Delta)
  :=\Delta^2+\sum_{j=1}^{s}a_j^2K_{\cF}\cdot C_j.
\]
\end{definition}
Using $K_{\cF}\cdot C_j+C_j^2=\tang(\cF,C_j)$, one obtains
\begin{equation}
  \tau(\cF,\Delta)
  =\sum_{j=1}^{s}a_j^2\tang(\cF,C_j)
   +2\sum_{i<j}a_ia_jC_i\cdot C_j.
  \label{eq:weighted-tangency-expansion}
\end{equation}
In particular,
\begin{equation}
  \tang(\cF,\Delta_{\mathrm{red}})=0
  \quad\Longrightarrow\quad
  \tau(\cF,\Delta)=0.
  \label{eq:weighted-tangency-vanishes}
\end{equation}
\begin{definition}[Assembled volume term]
\label{def:assembled-volume}
Define
\[
  \sV(\cF,\Delta)
  :=\vol(K_{\cF}+\Delta)-\tau(\cF,\Delta).
\]
Equivalently,
\begin{equation}
  \sV(\cF,\Delta)
  =K_{\cF}^2
   +\sum_{j=1}^{s}a_j(2-a_j)K_{\cF}\cdot C_j
   -N(\Delta)^2.
  \label{eq:assembled-volume-expanded}
\end{equation}
\end{definition}
By \eqref{eq:weighted-tangency-vanishes}, one has
\[
  \sV(\cF,\Delta)=\vol(K_{\cF}+\Delta)
\]
whenever $(X,\cF,\Delta)$ is tangency-free.

Let
\[
  \varrho_p\colon(X',\cF',\Delta')
  \longrightarrow(X,\cF,\Delta)
\]
be one cluster on its starting triple.  Write
\[
  K_{\cF}+\Delta=P(\Delta)+N(\Delta),
  \qquad
  K_{\cF'}+\Delta'=P(\Delta')+N(\Delta')
\]
for the Zariski decompositions before and after the cluster.
\begin{definition}[Adjoint Zariski index]
\label{def:adjoint-zariski-index}
The \emph{adjoint Zariski index} of the cluster over $p$ is
\[
  \alpha_\Delta(p)
  :=N(\Delta)^2-N(\Delta')^2.
\]
\end{definition}
The subscript $\Delta$ emphasizes that the index depends on the coefficients of
the boundary.  When the boundary is fixed, we occasionally write $\alpha(p)$.
\begin{definition}[Local volume correction]
\label{def:local-volume-correction}
For a cluster rooted at $p$, define
\begin{equation}
  T_1(p)
  :=\alpha_\Delta(p)
   +(1-\ell_p)
     \left(
       \sum_{j=1}^{s}a_j(2-a_j)\mult_p(C_j)-1
     \right).
  \label{eq:T1-definition}
\end{equation}
All curves and multiplicities in this formula are taken on the starting triple
of the cluster.
\end{definition}
\begin{lemma}[One-blow-up identity]
\label{lem:one-blowup-weighted-identity}
Let $r_j:=\mult_p(C_j)$ and $s_p:=1-\ell_p$.  Under the blow-up
\eqref{eq:one-blowup-adjoint}, one has
\begin{equation}
  (s_p-\mu_p)^2
  =\tau(\cF,\Delta)-\tau(\cF_1,\Delta_1)
   +s_p\left(1-\sum_{j=1}^{s}a_j(2-a_j)r_j\right).
  \label{eq:one-blowup-weighted-identity}
\end{equation}
\end{lemma}
\begin{proof}
Since
\[
  K_{\cF_1}=\sigma^*K_{\cF}+s_pE,
  \qquad
  C_{j,1}=\sigma^*C_j-r_jE,
\]
we have
\[
  \tau(\cF,\Delta)-\tau(\cF_1,\Delta_1)
  =\mu_p^2-s_p\sum_{j=1}^{s}a_j^2r_j.
\]
Now use $s_p^2=s_p$ and $\mu_p=\sum_ja_jr_j$.
\end{proof}
\begin{proposition}[Cluster transformation rule]
\label{prop:cluster-transformation-rule}
For every cluster rooted at $p$,
\begin{equation}
  \sV(\cF',\Delta')
  =\sV(\cF,\Delta)+T_1(p).
  \label{eq:cluster-transformation-rule}
\end{equation}
\end{proposition}
\begin{proof}
For a single blow-up with discrepancy coefficient $s_q-\mu_q$, the Zariski
decomposition gives
\[
  \vol(K_{\cF_1}+\Delta_1)-\vol(K_{\cF}+\Delta)
  =-(s_q-\mu_q)^2
   +N(\Delta)^2-N(\Delta_1)^2.
\]
Sum this equality and
Lemma~\ref{lem:one-blowup-weighted-identity} over all centers of the cluster.
By \eqref{eq:cluster-later-centers-singular}, $s_q=0$ at every center except
possibly the root.  The negative-part squares telescope to
$\alpha_\Delta(p)$, and the remaining root term is precisely the second term
in \eqref{eq:T1-definition}.
\end{proof}
\begin{remark}
The one-blow-up identity and the cluster transformation rule are formal
whenever the adjoint divisors at the two endpoints are pseudo-effective.
The hypothesis $\nu(\cF)\geq0$ is used to preserve irredundancy along the
cluster-adapted resolution and to obtain the nonnegative decomposition of
the assembled term below.
\end{remark}
\begin{proposition}[Positivity of the assembled term]
\label{prop:assembled-volume-nonnegative}
Let $(X,\cF,\Delta)$ be an irredundant triple with $\nu(\cF)\geq0$,
and write
\[
  K_{\cF}=P+N
\]
for the Zariski decomposition of $K_{\cF}$.  Then
\begin{equation}
  \sV(\cF,\Delta)
  =P^2+\sum_{j=1}^sa_j(2-a_j)P\cdot C_j+R(\cF,\Delta),
  \label{eq:assembled-volume-positive-decomposition}
\end{equation}
where
\[
  R(\cF,\Delta)
  :=\sum_{j=1}^sa_j(2-a_j)N\cdot C_j
     +N^2-N(\Delta)^2\geq0.
\]
In particular,
\[
  \sV(\cF,\Delta)\geq0.
\]
\end{proposition}
\begin{proof}
Since $P\cdot N=0$, one has $K_{\cF}^2=P^2+N^2$.
Substituting $K_{\cF}=P+N$ into
\eqref{eq:assembled-volume-expanded} gives
\eqref{eq:assembled-volume-positive-decomposition}.  The nonnegativity of
$R(\cF,\Delta)$ is Corollary~\ref{cor:R(F,D)}.
\end{proof}

\subsection{The negative part across a cluster}
\label{subsec:negative-part-across-cluster}
\subsubsection{Simple regular clusters}
\label{subsubsec:simple-regular-clusters}
We first treat the explicit class of clusters.  Let
$p\in\Supp\Delta$ be a simple regular initial center, that is,
\[
  t_p>0,\qquad \ell_p=0,\qquad m_p=1.
\]
Set $n=t_p+1$ and consider the sequence of blow-ups
\[
  \varrho_p\colon
  (X',\cF',\Delta')
  =(X_n,\cF_n,\Delta_n)
  \xrightarrow{\sigma_n}
  \cdots
  \xrightarrow{\sigma_2}
  (X_1,\cF_1,\Delta_1)
  \xrightarrow{\sigma_1}
  (X_0,\cF_0,\Delta_0)
  =(X,\cF,\Delta),
\]
where $\sigma_i$ is the blow-up at $q_i$, with exceptional divisor
$E_i$, and
\[
  q_1=p,\qquad
  q_k=E_{k-1}\cap\Delta_{k-1,\red},
  \quad k=2,\ldots,n.
\]

Let $C$ be the unique component of $\Delta_{\red}$ through $p$, and
let $C'$ be its strict transform on $X'$.  If $\bar E_i$ denotes the
strict transform of $E_i$ on $X'$, then the configuration over $p$ is
\[
  \underbrace{
    \bar E_1\,(-2)
    \;-\;\cdots\;-\;
    \bar E_{n-1}\,(-2)
  }_{[2,\ldots,2]}
  \;-\;
  E_n\,(-1)
  \;-\;
  C'.
\]
Moreover,
\[
  \ell(q_1)=0,
  \qquad
  \ell(q_i)=1
  \quad\text{for }i=2,\ldots,n.
\]

\begin{proposition}[Simple regular clusters]\label{prop:simple-regular-clusters}
Suppose \(t_p>0\), \(\ell_p=0\), \(m_p=1\), and the triple
\((X,\cF,\Delta)\) is irredundant with \(\nu(\cF)\geq0\).  With the
notation above, the following hold.
\begin{enumerate}
\item \(p\) does not lie on $\Supp N(\Delta)$.  During the morphism \(\varrho_p\), one has
\[
t_{q_1}=t_p,\qquad
t_{q_k}-t_{q_{k+1}}=
\begin{cases}
0,\quad&\text{for }k=1,\\
1,\quad&\text{for }k=2,\dots,t_p,
\end{cases}
\qquad t_{q_n}=1, 
\]
where $t_{q_i}:=\tang(\cF_{i-1},\Delta_{i-1,\red},q_i)$, $i=1,\dots,n$, and
\[
\tang(\cF',\Delta'_{\red},q)=0
\quad
\text{for every }q\in \Supp\Delta'\cap\varrho_p^{-1}(p).
\]
Moreover, \(\varrho_p\) is the cluster over \(p\).
\item The triple \((X', \cF', \Delta')\) is irredundant and
$K_{\cF'}$ is pseudo-effective.
\item The divisor
\[
\Theta^+=\bar E_1+\cdots+\bar E_{t_p}
\]
is a new maximal \((\Delta',\cF')\)-chain, and
\[
N(\Delta')=\varrho_p^*N(\Delta)+M(\Delta',\Theta^+),\qquad M(\Delta',\Theta^+)
=
\sum_{i=1}^{t_p}\frac{t_p+1-i}{t_p+1}\,\bar E_i.
\]
\end{enumerate}
\end{proposition}

\begin{proof}
Let $C$ be the unique component of $\Delta_{\mathrm{red}}$ through $p$,
and let $a$ be its coefficient in $\Delta$.
By the negative-part theorem, $\Supp N(\Delta)$ is the disjoint union of
the maximal $(\Delta,\cF)$-chains.  If $p$ lies on a component $\Gamma$
of such a chain, then $\Gamma$ is the leaf through the regular point
$p$.  Since $m_p=1$ and $t_p>0$, the curve $C$ is smooth at $p$ and
\[
  (C\cdot\Gamma)_p=t_p+1\geq2.
\]
Consequently,
\[
  \Delta\cdot\Gamma
  \geq a(C\cdot\Gamma)_p
  \geq2a
  \geq1,
\]
contradicting the defining intersection conditions of a
$(\Delta,\cF)$-chain.  Hence
\[
  p\notin\Supp N(\Delta).
\]
At the first center one has $\ell(q_1)=0$ and $m_{q_1}=1$, so the
one-blow-up tangency formula gives no tangency drop.  At every subsequent
center one has $\ell(q_i)=m_{q_i}=1$, and the tangency drops by one.
This gives all the displayed identities for the $t_{q_i}$.  In
particular, $t_{q_n}=1$, and the blow-up at $q_n$ eliminates the remaining
tangency over $p$.  Since the tangency is positive before the final
blow-up, the sequence is minimal.  Thus $\varrho_p$ is the cluster over
$p$, proving \rm(1).
Assertion \rm(2) follows from
Proposition~\ref{prop:cluster-preserves-irredundancy}.
For \rm(3), the local exceptional configuration shows that
\[
  \Theta^+=\bar E_1+\cdots+\bar E_{t_p}
\]
is a new $(\Delta',\cF')$-chain of type $[2,\ldots,2]$.  Since
$p\notin\Supp N(\Delta)$, it cannot join the strict transform of an old
maximal chain, and is therefore maximal.  Moreover, $\Delta'$ meets only
the last exceptional curve $E_{t_p+1}$, so
\[
  \Delta'\cdot\bar E_i=0
  \qquad(1\leq i\leq t_p).
\]
Hence the boundary multiplicity of $\Theta^+$ is zero, and solving the
intersection equations for the string $[2,\ldots,2]$ gives the displayed
formula for $M(\Theta^+)$.
Finally, all old maximal $(\Delta,\cF)$-chains are unchanged because
$p\notin\Supp N(\Delta)$.  By \rm(2), the negative-part theorem applies
also to the output triple, whose maximal chains are the strict transforms
of the old ones together with $\Theta^+$.  This gives the stated formula
for $N(\Delta')$.
\end{proof}
\subsubsection{Singular and multiple regular clusters}
\label{subsubsec:singular-multiple-regular-clusters}
We now treat the remaining two types of clusters.  In contrast with a simple
regular cluster, the exceptional configuration is not determined solely by
the tangency order at the root: an old maximal $(\Delta,\cF)$-chain may be
replaced or enlarged.  For a smooth rational $\cF$-invariant curve $\Gamma$
through a point $p$, set
\[
  d_p(\Gamma):=(\Delta\cdot\Gamma)_p.
\]
For a $(\Delta,\cF)$-chain
$\Theta=\Gamma_1+\cdots+\Gamma_r$, we write
\[
  \theta_\Delta(\Theta):=\Delta\cdot\Gamma_1.
\]
\begin{definition}[Potential curves]
\label{def:potentialcurve}
Let $p\in\Supp\Delta$ be a singular cluster root, so that $\ell_p=1$.
A smooth rational $\cF$-invariant curve $\Gamma$ through $p$ is called a
\emph{potential curve} of $(\Delta,\cF)$ if one of the following conditions
holds.
\begin{enumerate}
  \item[\rm(E1)] $\Theta=\Gamma$ is a one-component maximal
  $(\Delta,\cF)$-chain, and $p$ is the unique nondegenerate reduced
  singularity of $\cF$ on $\Gamma$.
  \item[\rm(E2)] The point $p$ is the unique nondegenerate reduced
  singularity of $\cF$ on $\Gamma$, and
  \[
    1\leq\Delta\cdot\Gamma<d_p(\Gamma)+1.
  \]
  \item[\rm(E3)] The curve $\Gamma$ contains precisely two nondegenerate
  reduced singularities $p$ and $q$, meets a maximal
  $(\Delta,\cF)$-chain
  \[
    \Theta=\Gamma_1+\cdots+\Gamma_r
  \]
  at $q=\Gamma\cap\Gamma_r$, and
  \[
    \Delta\cdot\Gamma=d_p(\Gamma).
  \]
\end{enumerate}
\end{definition}
We now describe the change of the negative part across a complete cluster.
\begin{proposition}[Singular or multiple regular clusters]
\label{prop:singular-or-multiple-regular-centers}
Let
\[
  \varrho_p\colon(X',\cF',\Delta')\longrightarrow(X,\cF,\Delta)
\]
be a complete cluster starting from an irredundant triple with
$\nu(\cF)\geq0$.
Assume that either $\ell_p=0$ and $m_p\geq2$, or $\ell_p=1$ and
$m_p\geq1$.  Then $(X',\cF',\Delta')$ is irredundant and
$K_{\cF'}$ is pseudo-effective, and the following statements hold.
\begin{enumerate}
  \item Suppose that $\ell_p=0$.  The old maximal
  $(\Delta,\cF)$-chains are unchanged.  At most one new maximal
  $(\Delta',\cF')$-chain $\Theta^+$ can occur; if it occurs, it is contained
  in $\varrho_p^{-1}(p)$ and its first component is the strict transform
  $\bar E_1$ of the first exceptional divisor.  More precisely,
  \[
    N(\Delta')=
    \begin{cases}
      \varrho_p^*N(\Delta),
        &\Delta'\cdot\bar E_1\geq1,\\[1mm]
      \varrho_p^*N(\Delta)+M(\Delta',\Theta^+),
        &\Delta'\cdot\bar E_1<1.
    \end{cases}
  \]
  In the second case,
  \[
    \theta_{\Delta'}(\Theta^+)=\Delta'\cdot\bar E_1.
  \]

  \item Suppose that $\ell_p=1$.  Then
\[
  N(\Delta')=\varrho_p^*N(\Delta)
\]
unless $p$ lies on a potential curve $\Gamma$.  In each exceptional
case below, write $\Xi$ for the exceptional part of $\Theta^+$,
allowing $\Xi=0$.  Thus $\Xi$ is a Hirzebruch--Jung string supported
on a single branch of $\varrho_p^{-1}(p)$.

\begin{enumerate}

  \item[\rm(E1)] The strict transform $\bar\Gamma$ of the old
  one-component chain $\Theta=\Gamma$ extends to the maximal chain
  $\Theta^+=\bar\Gamma+\Xi$.
  Moreover,
  \[
    N(\Delta')
    =\varrho_p^*\!\bigl(N(\Delta)-M(\Delta,\Theta)\bigr)
     +M(\Delta',\Theta^+),
    \qquad
    \theta_{\Delta'}(\Theta^+)=0.
  \]

  \item[\rm(E2)] A new maximal $(\Delta',\cF')$-chain
  $\Theta^+=\bar\Gamma+\Xi$
  is created, and
  \[
    N(\Delta')
    =\varrho_p^*N(\Delta)+M(\Delta',\Theta^+),
    \qquad
    \theta_{\Delta'}(\Theta^+)
    =\Delta\cdot\Gamma-d_p(\Gamma).
  \]

  \item[\rm(E3)] Let $\bar\Theta$ be the strict transform of the old
  maximal chain $\Theta$.  Then $\Theta$ is enlarged to
  $\Theta^+=\bar\Theta+\bar\Gamma+\Xi$,
  and
  \[
    N(\Delta')
    =\varrho_p^*\!\bigl(N(\Delta)-M(\Delta,\Theta)\bigr)
     +M(\Delta',\Theta^+),
    \qquad
    \theta_{\Delta'}(\Theta^+)=\theta_\Delta(\Theta).
  \]
\end{enumerate}
\end{enumerate}
\end{proposition}

\begin{proof}

The birational transformation formula and $\nu(\cF)\geq0$ imply that
$K_{\cF'}$ is pseudo-effective.  By
Proposition~\ref{prop:cluster-preserves-irredundancy}, the output
triple is irredundant.  Hence the negative-part theorem applies at
both endpoints.

Suppose first that $\ell_p=0$.  Since $m_p\geq2$ and every coefficient
of $\Delta$ is at least $1/2$, one has $\mu_p\geq1$.  If $p$ lay on
a maximal $(\Delta,\cF)$-chain, then its invariant component through
$p$ would be the leaf through $p$ and would have boundary intersection
at least $\mu_p$, contradicting the defining intersection conditions
of a $(\Delta,\cF)$-chain.  Thus
$p\notin\Supp N(\Delta)$,
and every old maximal chain remains unchanged.

Among the exceptional curves over $p$, only $\bar E_1$ can be the
first component of a new chain, since
\[
  K_{\cF'}\cdot\bar E_1=-1,
  \qquad
  K_{\cF'}\cdot E=0
\]
for every other exceptional component $E$.  Nor can the first
component be the strict transform $\bar C$ of a curve $C\subset X$.
Indeed, $C$ would be the invariant leaf through $p$.  The singularity
of $\cF'$ on $\bar C$ over $p$ accounts for
$\mathrm Z(\cF',\bar C)=1$, so $C$ contains no singularity of
$\cF$.  The invariant index and Camacho--Sad formulas then give
\[
  K_{\cF}\cdot C=-2,
  \qquad
  C^2=0.
\]
Thus $C$ is nef, contradicting the pseudo-effectivity of $K_{\cF}$.
Hence every new chain begins with $\bar E_1$.

Such a chain cannot contain any curve coming from $X$.  Indeed, the
only possible one is the invariant separatrix through $p$, and the
exceptional path joining its strict transform to $\bar E_1$ contains
a $(-1)$-curve, which cannot belong to a Hirzebruch--Jung string.
Consequently, every new maximal chain is contained in
$\varrho_p^{-1}(p)$ and begins with $\bar E_1$.  It occurs precisely
when $\Delta'\cdot\bar E_1<1$.
This proves \rm(1).

Suppose now that $\ell_p=1$.  Since every center of the cluster is
singular,
\[
  K_{\cF'}=\varrho_p^*K_{\cF}.
\]
Thus every exceptional component has canonical intersection zero and
cannot be the first component of an $\cF'$-chain.  Therefore the first
component of every maximal chain affected by the cluster comes from
$X$.

Let $\Theta=\Gamma_1+\cdots+\Gamma_r$ be an old maximal
$(\Delta,\cF)$-chain passing through $p$.  Since
$\Delta\cdot\Gamma_i=0$ for $i\geq2$, one has $p\in\Gamma_1$.  If
$r\geq2$, then $p$ must be the unique singularity
$\Gamma_1\cap\Gamma_2$ of $\cF$ on $\Gamma_1$, forcing
$\Delta\cdot\Gamma_2>0$, a contradiction.  Hence $r=1$; write
$\Theta=\Gamma$.

Since $\Delta\cdot\Gamma<1$ and every positive coefficient of
$\Delta$ is at least $1/2$, the boundary meets $\Gamma$ only at $p$
and does so transversely.  Hence
$\Delta'\cdot\bar\Gamma=0$.

No other component of the resulting maximal chain can come from
$X$.  Indeed, since $p$ is the unique singularity of $\cF$ on
$\Gamma$, any such component would have to be the strict transform
of the other invariant separatrix through $p$.  The exceptional path
joining it to $\bar\Gamma$ contains a $(-1)$-curve, which cannot
belong to a Hirzebruch--Jung string.  Thus the remaining components
form a possibly empty Hirzebruch--Jung string $\Xi$ contained in a
single branch of $\varrho_p^{-1}(p)$, and
\[
  \Theta^+=\bar\Gamma+\Xi,
  \qquad
  \theta_{\Delta'}(\Theta^+)=0.
\]
This is case \rm(E1).

Assume now that no old maximal chain passes through $p$.  If no
maximal chain is affected by the cluster, then
\[
  N(\Delta')=\varrho_p^*N(\Delta).
\]
Otherwise, let $\Theta^+$ be an affected maximal chain.

If its first component does not come from an old maximal chain, then
it is the strict transform $\bar\Gamma$ of a smooth rational
$\cF$-invariant curve $\Gamma$ through $p$.  Since
\[
  K_{\cF}\cdot\Gamma
  =K_{\cF'}\cdot\bar\Gamma=-1,
\]
one has $\mathrm Z(\cF,\Gamma)=1$, so $p$ is the unique
nondegenerate singularity of $\cF$ on $\Gamma$.  Moreover,
\[
  \Delta'\cdot\bar\Gamma
  =\Delta\cdot\Gamma-d_p(\Gamma)<1.
\]
The Camacho--Sad formula gives $\Gamma^2\leq-1$.  If
$\Delta\cdot\Gamma<1$, then $\Gamma$ would be contained in an old
maximal $(\Delta,\cF)$-chain when $\Gamma^2\leq-2$, and would be
redundant when $\Gamma^2=-1$.  Both are impossible.  Therefore
\[
  1\leq\Delta\cdot\Gamma<d_p(\Gamma)+1.
\]
The same argument as in case \rm(E1) shows that no other component of
$\Theta^+$ comes from $X$.  Hence
\[
  \Theta^+=\bar\Gamma+\Xi,
\]
where $\Xi$ is a possibly empty Hirzebruch--Jung string contained in
a single branch of $\varrho_p^{-1}(p)$.  This is case \rm(E2).

In the remaining case, $\Theta^+$ enlarges an old maximal chain
$\Theta=\Gamma_1+\cdots+\Gamma_r$.  There is a smooth rational
$\cF$-invariant curve $\Gamma$ through $p$ joining $\Gamma_r$ to the
exceptional fiber.  Since $\bar\Gamma$ is a non-first component of
$\Theta^+$,
\[
  K_{\cF}\cdot\Gamma
  =K_{\cF'}\cdot\bar\Gamma=0,
\]
and hence $\mathrm Z(\cF,\Gamma)=2$.  Its two singularities are
therefore $p$ and $\Gamma\cap\Gamma_r$.  Moreover,
$\Delta'\cdot\bar\Gamma=0$, so
\[
  \Delta\cdot\Gamma=d_p(\Gamma).
\]

By the maximality of $\Theta$ and the same argument as in
case \rm(E1), no additional component of $\Theta^+$ comes from $X$.
Writing $\bar\Theta$ for the strict transform of $\Theta$, one has
\[
  \Theta^+=\bar\Theta+\bar\Gamma+\Xi,
\]
where $\Xi$ is a possibly empty Hirzebruch--Jung string contained in
a single branch of $\varrho_p^{-1}(p)$.  Since the first component of
the old chain is unchanged,
\[
  \theta_{\Delta'}(\Theta^+)=\theta_\Delta(\Theta).
\]
This is case \rm(E3), and the three cases are exhaustive.

The asserted formulas now follow from the decomposition of the
negative part into the contributions of its maximal chains.  In
part \rm(1), the old contributions remain unchanged and
$M(\Delta',\Theta^+)$ is the only possible new contribution.  In
cases \rm(E1) and \rm(E3), the old contribution is replaced by
$M(\Delta',\Theta^+)$, while in case \rm(E2) it is added as a new
contribution.

\end{proof}

\begin{remark}
The simple regular clusters of
Proposition~\ref{prop:simple-regular-clusters} formally fit the second
alternative in part~\rm(1): they create a chain over $p$ with first component
$\bar E_1$ and boundary multiplicity zero.  They are treated separately
because the resulting maximal chain is explicit, namely a
$[2,\ldots,2]$-string.
\end{remark}

\subsection{Adjoint Zariski indices of clusters}
\label{subsec:adjoint-zariski-index-computation}
The preceding subsection determines the change of the negative part across
each type of cluster.  We now express its numerical consequence in terms of
the local $\beta$-numbers of the singularities carried by the maximal chains.
Let $q$ be a nondegenerate reduced singularity lying on an
$\cF$-chain.  Write its characteristic ratio as
\[
  -\lambda_q=\frac{u_q}{v_q}\in\mathbb Q_{>0},
  \qquad
  \gcd(u_q,v_q)=1,
\]
and set
\[
  \beta_q(\cF):=\frac{1}{u_qv_q}.
\]
This definition is independent of the ordering of the two
eigenvalues.
For an $\cF$-chain $\Theta$, the invariant
$\beta_{\cF}(\Theta)$ introduced in
Subsection~\ref{subsec:special-chains} admits the local expression
\begin{equation}
  \beta_{\cF}(\Theta)
  =
  \sum_{q\in\Theta\cap\Sing(\cF)}\beta_q(\cF).
  \label{eq:chain-beta-mass}
\end{equation}

Recall from \eqref{eq:MTheta-square} that, for
$\Theta=\Gamma_1+\cdots+\Gamma_r$ with $e_i=-\Gamma_i^2$,
\[
  n_\Theta=[e_1,\ldots,e_r],
  \qquad
  q_\Theta=[e_2,\ldots,e_r],
  \qquad
  \beta_{\cF}(\Theta)=\frac{q_\Theta}{n_\Theta},
  \qquad
  \coeff_{\Gamma_r}M(\Theta)=\frac1{n_\Theta},
\]
where $q_\Theta=1$ if $r=1$.
In particular, the $\beta$-mass and the coefficient at the attaching end
of the chain are controlled by the same index $n_\Theta$.
We record the following elementary index identity.  Suppose
that a chain $\Theta^+$ is obtained from $\Theta$ by attaching a nonempty
tail
\[
  D_1+\cdots+D_s,
  \qquad d_j:=-D_j^2,
\]
at $\Gamma_r$.  Put
\[
  n_+:=n_{\Theta^+},\qquad q_+:=q_{\Theta^+},\qquad
  c(\Theta^+/\Theta):=[d_2,\ldots,d_s],
\]
where the last expression is $1$ if $s=1$.  
A direct calculation gives
\begin{equation}
  n_\Theta q_+-q_\Theta n_+
  =c(\Theta^+/\Theta).
  \label{eq:index-tail-determinant-identity}
\end{equation}
Consequently,
\begin{equation}
  \beta_{\cF'}(\Theta^+)-\beta_{\cF}(\Theta)
  =\frac{q_+}{n_+}-\frac{q_\Theta}{n_\Theta}
  =\frac{c(\Theta^+/\Theta)}{n_\Theta n_+}>0.
  \label{eq:beta-increment-by-old-chain-index}
\end{equation}
Notice that the numerator only records the tail added by the cluster,
whereas the denominator contains the index $n_\Theta$ of the original
chain.
For a $(\Delta,\cF)$-chain $\Theta$, its
\emph{boundary-weighted $\beta$-mass} is
\begin{equation}
  \beta_\Delta(\Theta)
  :=\bigl(1-\theta_\Delta(\Theta)\bigr)^2
    \beta_{\cF}(\Theta)
  =-M(\Delta,\Theta)^2.
  \label{eq:boundary-weighted-beta-square}
\end{equation}
Let
\[
  \varrho_p\colon(X',\cF',\Delta')\longrightarrow(X,\cF,\Delta)
\]
be a cluster.  Let $M_p^-$ and $M_p^+$ be, respectively, the maximal-chain
contributions removed from the old negative part and inserted in the new one;
either is set equal to zero when absent.  The formulas of
Propositions~\ref{prop:simple-regular-clusters} and
\ref{prop:singular-or-multiple-regular-centers} take the uniform form
\[
  N(\Delta')
  =\varrho_p^*\bigl(N(\Delta)-M_p^-\bigr)+M_p^+.
\]
The maximal-chain decomposition gives
\[
  \bigl(N(\Delta)-M_p^-\bigr)\cdot M_p^-=0,
  \qquad
  \varrho_p^*\bigl(N(\Delta)-M_p^-\bigr)\cdot M_p^+=0.
\]
Consequently,
\begin{equation}
  \alpha_\Delta(p)
  =(M_p^-)^2-(M_p^+)^2.
  \label{eq:adjoint-zariski-index-uniform}
\end{equation}
Thus $\alpha_\Delta(p)$ is the inserted boundary-weighted $\beta$-mass
minus the removed boundary-weighted $\beta$-mass.
\begin{proposition}[Adjoint Zariski indices]
\label{prop:adjoint-zariski-indices}
For every cluster root $p$, the following hold.
\begin{enumerate}
  \item If $p$ is simple regular, then $M_p^-=0$ and
  $M_p^+=M(\Delta',\Theta^+)$.  Hence
  \begin{equation}
    \alpha_\Delta(p)
    =-M(\Delta',\Theta^+)^2
    =\frac{q_+}{n_+}
    =\frac{t_p}{t_p+1},
    \qquad (n_+,q_+)=(t_p+1,t_p).
    \label{eq:alpha-simple-regular}
  \end{equation}
  \item Suppose that $p$ is multiple regular.  If the cluster creates no
  maximal chain, then $\alpha_\Delta(p)=0$.  Otherwise,
  \[
    \alpha_\Delta(p)
    =\beta_{\Delta'}(\Theta^+)
    =\bigl(1-\theta_{\Delta'}(\Theta^+)\bigr)^2
      \frac{q_+}{n_+}>0,
  \]
  where $n_+=n_{\Theta^+}$ and $q_+=q_{\Theta^+}$.
  \item Suppose that $p$ is singular.  If no potential curve occurs, then
  $\alpha_\Delta(p)=0$.  In the three exceptional cases, the index is as
  follows.
  \begin{enumerate}
    \item[\rm(E1)] Write $\Theta=\Gamma$,
    $n_\Theta=e=-\Gamma^2\geq2$, and
    $\theta=\Delta\cdot\Gamma\in[1/2,1)$.  If
    $n_+=n_{\Theta^+}$ and $q_+=q_{\Theta^+}$, then
    \begin{equation}
      \begin{aligned}
      \alpha_\Delta(p)
      &=\frac{q_+}{n_+}
        -\frac{(1-\theta)^2}{n_\Theta}\\
      &\geq\frac{1}{n_\Theta+1}
        -\frac{(1-\theta)^2}{n_\Theta}
       \geq\frac{3n_\Theta-1}
                    {4n_\Theta(n_\Theta+1)}>0.
      \end{aligned}
      \label{eq:alpha-E1-beta}
    \end{equation}
    \item[\rm(E2)] Put
    \[
      \theta^+
      :=\theta_{\Delta'}(\Theta^+)
      =\Delta\cdot\Gamma-d_p(\Gamma)\in[0,1).
    \]
    Since no old chain contribution is removed,
    \[
      \alpha_\Delta(p)
      =\beta_{\Delta'}(\Theta^+)
      =(1-\theta^+)^2\frac{q_+}{n_+}>0,
    \]
    where $n_+=n_{\Theta^+}$ and $q_+=q_{\Theta^+}$.
    \item[\rm(E3)] Let
    $\Theta=\Gamma_1+\cdots+\Gamma_r$ be the old maximal chain, put
    $e_i=-\Gamma_i^2$, and set $\theta:=\theta_\Delta(\Theta)$. 
    Let $\bar\Theta$ be the strict transform of $\Theta$.  Write
\[
  \Theta^+=\bar\Theta+D_1+\cdots+D_s,
  \qquad
  d_j:=-D_j^2,
\]
    in the order of the chain, and put
    \[
      n_\Theta=[e_1,\ldots,e_r],\qquad
      n_+=n_{\Theta^+},\qquad
      c_p:=[d_2,\ldots,d_s],
    \]
    with $c_p=1$ if $s=1$.  Then
    \begin{equation}
      \begin{aligned}
      \alpha_\Delta(p)
      &=(1-\theta)^2
        \bigl(\beta_{\cF'}(\Theta^+)-\beta_{\cF}(\Theta)\bigr)\\
      &=(1-\theta)^2\frac{c_p}{n_\Theta n_+}>0.
      \end{aligned}
      \label{eq:alpha-E3-beta}
    \end{equation}
    Equivalently, if $\mathcal Q_p$ denotes the singularities of
$\Theta^+$ lying on the added tail away from its attaching point to
$\bar\Theta$, then
    \[
      \sum_{q\in\mathcal Q_p}\beta_q(\cF')
      =\frac{c_p}{n_\Theta n_+},
      \qquad
      \alpha_\Delta(p)
      \geq\frac{(1-\theta)^2}{n_\Theta n_+}.
    \]
  \end{enumerate}
\end{enumerate}
In particular,
\[
  \alpha_\Delta(p)\geq0
  \qquad\text{for every }p\in\sS_\rho.
\]
\end{proposition}
\begin{proof}
Formula~\eqref{eq:adjoint-zariski-index-uniform} gives all the asserted
identities after applying
\eqref{eq:boundary-weighted-beta-square}.  For a simple regular cluster,
the $[2,\ldots,2]$-string computed
in Proposition~\ref{prop:simple-regular-clusters} satisfies
\[
  \beta_{\cF'}(\Theta^+)
  =-M(\Delta',\Theta^+)^2
  =\frac{t_p}{t_p+1}.
\]
The multiple regular case and case \rm(E2) follow because a nonempty
Hirzebruch--Jung chain has positive $\beta$-mass.
In case \rm(E1), the boundary assumptions imply that it meets $\Gamma$
only at $p$ and transversely.  Hence $\bar\Gamma\cap E_1$ is not a later
cluster center, and the new chain begins with
$\bar\Gamma^2=-(n_\Theta+1)$.  The continuant recursion gives
\[
  n_+\leq(n_\Theta+1)q_+,
  \qquad\text{hence}\qquad
  \frac{q_+}{n_+}\geq\frac1{n_\Theta+1}.
\]
Together with $(1-\theta)^2\leq1/4$, this proves
\eqref{eq:alpha-E1-beta}.
Finally, in case \rm(E3), the old singularities and the boundary weight are
unchanged.  Hence \eqref{eq:chain-beta-mass} and
\eqref{eq:beta-increment-by-old-chain-index} give
\eqref{eq:alpha-E3-beta} and its equivalent local formula involving
$\mathcal Q_p$.
\end{proof}
\begin{remark}[The one-blow-up subcase]
Suppose that a singular cluster consists of a single blow-up, and set
$e=-\Gamma^2$, where $\Gamma$ is the corresponding potential curve.
In case \rm(E3), write the old chain as
$\Theta=\Gamma_1+\cdots+\Gamma_r$, set $e_i=-\Gamma_i^2$, let
$n_\Theta$ be its index, and put
$n_+=n_{\Theta^+}=[e_1,\ldots,e_r,e+1]$.  The preceding formulas reduce
respectively to
\[
  \begin{array}{ll}
  \mathrm{(E1)}&\displaystyle
  \alpha_\Delta(p)
  =\frac{1}{n_+}-\frac{(1-\theta)^2}{n_\Theta},
  \quad (n_\Theta,n_+)=(e,e+1),\\[3mm]
  \mathrm{(E2)}&\displaystyle
  \alpha_\Delta(p)
  =\frac{(1-\theta^+)^2}{n_+},
  \quad n_+=e+1,\\[3mm]
  \mathrm{(E3)}&\displaystyle
  \alpha_\Delta(p)
  =\frac{(1-\theta)^2}{n_\Theta n_+}.
  \end{array}
\]
These are precisely the one-blow-up formulas. 
\end{remark}
If $p$ is simple regular and $a$ is the coefficient of the unique component
of $\Delta$ through $p$, Definition~\ref{def:local-volume-correction} and
\eqref{eq:alpha-simple-regular} give
\begin{equation}
  T_1(p)=\frac{t_p}{t_p+1}-(1-a)^2.
  \label{eq:T1-simple-regular}
\end{equation}

\subsection{The global cluster formula and local positivity}
\label{subsec:global-cluster-formula}
We now sum the cluster corrections along a cluster-adapted resolution.
\begin{theorem}[Cluster formula for the volume]
\label{thm:cluster-volume-formula}
Let $(X_0,\cF_0,\Delta_0)$ be a triple such that
$(X_0,\cF_0)$ is relatively minimal and
$\nu(\cF_0)\geq0$.
Write
\[
  \Delta_0=\sum_{j=1}^s a_jC_{0,j},
\]
and let
\[
  \rho\colon(X,\cF,\Delta)
  \longrightarrow(X_0,\cF_0,\Delta_0)
\]
be a cluster-adapted resolution.  Then
\begin{equation}
  \vol(K_{\cF}+\Delta)
  =\sV(\cF_0,\Delta_0)
   +\sum_{p\in\sS_\rho}T_1(p).
  \label{eq:global-cluster-volume-formula}
\end{equation}
Equivalently,
\begin{equation}
  \begin{aligned}
  \vol(K_{\cF}+\Delta)
  ={}&K_{\cF_0}^2
     +\sum_{j=1}^sa_j(2-a_j)K_{\cF_0}\cdot C_{0,j}
     -N(\Delta_0)^2\\
    &+\sum_{p\in\sS_\rho}T_1(p).
  \end{aligned}
  \label{eq:global-cluster-volume-formula-expanded}
\end{equation}
\end{theorem}
\begin{proof}
Iterating Proposition~\ref{prop:cluster-transformation-rule} gives
\[
  \sV(\cF,\Delta)
  =\sV(\cF_0,\Delta_0)
   +\sum_{p\in\sS_\rho}T_1(p).
\]
Since $\tang(\cF,\Delta_{\red})=0$, one has
$\tau(\cF,\Delta)=0$ and therefore
$\sV(\cF,\Delta)=\vol(K_{\cF}+\Delta)$.  This proves
\eqref{eq:global-cluster-volume-formula}; the expanded formula follows from
\eqref{eq:assembled-volume-expanded}.
\end{proof}
\begin{corollary}[Positivity of the local terms]
\label{cor:T1-positive}
For every $p\in\sS_\rho$, one has $T_1(p)\geq0$.  More precisely:
\begin{enumerate}
  \item if $\ell_p=1$, then
  \[
    T_1(p)=\alpha_\Delta(p)\geq0.
  \]
  In particular, $T_1(p)>0$ if and only if a potential curve occurs.
  \item if $\ell_p=0$ and $m_p\geq2$, then
  \[
    T_1(p)\geq\frac34m_p-1\geq\frac12;
  \]
  \item if $p$ is simple regular, then
  \[
    T_1(p)=\frac{t_p}{t_p+1}-(1-a)^2\geq\frac14,
  \]
  where $a$ is the coefficient of the unique component of $\Delta$ through
  $p$.
\end{enumerate}
Consequently,
\[
  \vol(K_{\cF}+\Delta)
  \geq\sV(\cF_0,\Delta_0)
  \geq0.
\]
\end{corollary}
\begin{proof}
The first assertion follows from
Proposition~\ref{prop:adjoint-zariski-indices}.  For a multiple regular root,
$a_j(2-a_j)\geq3/4$ for every component through $p$, and hence
\[
  T_1(p)
  \geq\frac34\sum_j\mult_p(C_j)-1
  =\frac34m_p-1.
\]
For a simple regular root, use \eqref{eq:T1-simple-regular}, together with
$t_p\geq1$ and $a\geq1/2$.  The final inequalities follow from
Theorem~\ref{thm:cluster-volume-formula} and
Proposition~\ref{prop:assembled-volume-nonnegative}.
\end{proof}
\begin{remark}[Standard coefficients]
If $a_j=1-1/n_j$, then
\[
  a_j(2-a_j)=1-\frac1{n_j^2}
\]
and
\[
  T_1(p)
  =\alpha_\Delta(p)
   +(1-\ell_p)
    \left(
      \sum_j\left(1-\frac1{n_j^2}\right)\mult_p(C_j)-1
    \right).
\]
For a simple regular root lying on a component of coefficient $1-1/n$, this
reduces to
\[
  T_1(p)=\frac{t_p}{t_p+1}-\frac1{n^2}.
\]
\end{remark}

\section{Adjoint volume and birational automorphism bounds}
\label{sec:volume-and-automorphism-bounds}
In this section we prove the three cases of
Theorem~\ref{thm:adjoint-volume-bounds} and then apply them to the
quotient pair furnished by Corollary~\ref{cor:Galois-cover}.
Let $\cH$ be a canonical foliation on a smooth projective surface $Z$.
For $r\geq1$, recall that its $r$-th \emph{pluricanonical index} is
\begin{equation}
  \delta_r(\cH)
  :=
  \min\bigl\{
    m\in\mathbb Z_{>0}
    \mid h^0(Z,mK_{\cH})\geq r
  \bigr\},
  \label{eq:pluricanonical-section-index}
\end{equation}
where $\min\varnothing:=\infty$.  
(For a foliation that is not necessarily canonical, its
pluricanonical indices are computed on any canonical birational model.)
For the cases $\kappa(\cF)=0$ and $1$, let $(X,\cF,\Delta)$ be the
pair in Theorem~\ref{thm:adjoint-volume-bounds}.  
By the reduction at the beginning of
Section~\ref{sec:cluster-formula}, we may assume that it is the output
of a cluster-adapted resolution
\[
  \rho\colon
  (X,\cF,\Delta)
  \longrightarrow
  (X_0,\cF_0,\Delta_0),
  \qquad
  \Delta=\rho_*^{-1}\Delta_0,
\]
as in Theorem~\ref{thm:cluster-volume-formula}, where the target is
relatively minimal.  The pluricanonical indices and
Kodaira dimension are birationally invariant, so
\[
  \delta_r(\cF_0)=\delta_r(\cF),
  \qquad
  \kappa(\cF_0)=\kappa(\cF).
\]
We write
\[
  \Delta=\sum_i a_iC_i,
  \qquad
  \Delta_0=\sum_i a_iC_{0,i},
  \qquad
  C_{0,i}:=\rho_*C_i.
\]
\subsection{Kodaira dimension zero}
\label{subsec:kappa-zero-volume-bound}
We begin with the numerically trivial case.  It is the basic case to
which the general Kodaira-dimension-zero case will be reduced by the
canonical cyclic cover.
\begin{lemma}\label{lem:numerically-trivial-case}
Assume that $K_{\cF_0}\equiv0$. Then
\[
\vol(K_{\cF}+\Delta)\geq\frac14.
\]
\end{lemma}
\begin{proof}
Since $(X_0,\cF_0)$ is relatively minimal and $K_{\cF_0}\equiv0$, the
Zariski decomposition of $K_{\cF_0}$ satisfies
\[
P\equiv0,\qquad N=0.
\]
Since $N(\Delta_0)\leq N$, it follows that
\[
  N(\Delta_0)=0
  \qquad\text{and}\qquad
  \sV(\cF_0,\Delta_0)=0.
\]
Moreover, there are no potential curves.  Hence
Corollary~\ref{cor:T1-positive} gives
\[
  T_1(p)=0
\]
for every singular root $p$.
The cluster formula consequently gives
\[
\vol(K_{\cF}+\Delta)
=
\sum_{\substack{p\in\sS_\rho\\ \ell_p=0}}T_1(p).
\]
Since $K_{\cF}+\Delta$ is big, the sum is nonempty. Every
regular cluster satisfies
\[
T_1(p)\geq\frac14.
\]
Hence
\[
\vol(K_{\cF}+\Delta)\geq\frac14.
\]
\end{proof}

We now pass to the general Kodaira-dimension-zero case.  The loss in
the lower bound is measured precisely by the first
pluricanonical-section index.
\begin{proposition}\label{prop:kappa-zero-index-bound}
If $\kappa(\cF)=0$, then
\[
\vol(K_{\cF}+\Delta)
\geq
\frac{1}{4\delta_1(\cF)}.
\]
\end{proposition}
\begin{proof}
If $K_{\cF_0}\equiv0$, then
Lemma~\ref{lem:numerically-trivial-case} gives
\[
\vol(K_{\cF}+\Delta)
\geq\frac14
\geq\frac{1}{4\delta_1(\cF)}.
\]
Suppose that $K_{\cF_0}\not\equiv0$, and set
$h:=\delta_1(\cF)=\delta_1(\cF_0)\geq2$. Choose
\[
0\neq s\in H^0(X_0,hK_{\cF_0}),
\]
which is unique up to scalar, and let
\[
\pi\colon(Y,\cH)\longrightarrow(X_0,\cF_0)
\]
be the associated canonical cyclic cover of degree $h$, where $Y$ is
normal and $\cH$ is the saturated pullback of $\cF_0$.
Let $Y'$ be the normalization of $Y\times_{X_0}X$, and denote the
induced morphisms by
\[
\pi'\colon Y'\longrightarrow X,
\qquad
\mu\colon Y'\longrightarrow Y.
\]
Let
\[
\tau\colon(\widetilde Y,\widetilde\cH)\longrightarrow(Y',\cH')
\]
be the minimal resolution with the induced reduced foliation, and set
\[
\Delta_{Y'}:=(\pi')^*\Delta,
\qquad
\widetilde\Delta_Y:=\tau_*^{-1}\Delta_{Y'}.
\]
These morphisms fit into the commutative diagram
\[
\begin{tikzcd}[column sep=large,row sep=large]
(\widetilde Y,\widetilde\cH,\widetilde\Delta_Y)
  \arrow[r,"\tau"]
  \arrow[dr,"\pi'\circ\tau"']
&
(Y',\cH',\Delta_{Y'})
  \arrow[r,"\mu"]
  \arrow[d,"\pi'"]
&
(Y,\cH)
  \arrow[d,"\pi"]
\\
&
(X,\cF,\Delta)
  \arrow[r,"\rho"']
&
(X_0,\cF_0,\Delta_0).
\end{tikzcd}
\]
Write $K_{\cF_0}=P+N$ for the Zariski decomposition. Since
$\kappa(\cF_0)=0$, we have $P\equiv0$, and
\[
(s)_0=hN.
\]
Every nonzero coefficient of $N$ is strictly smaller than one.
Consequently, the reduced branch divisor of $\pi$ is $\Supp N$ and is
therefore $\cF_0$-invariant.
The same remains true after the normalized base change: every
ramification component of $\pi'$ is $\cF$-invariant. The foliated
Riemann--Hurwitz formula gives
\[
K_{\cH'}=(\pi')^*K_{\cF},
\]
and hence
\[
K_{\cH'}+\Delta_{Y'}
=
(\pi')^*(K_{\cF}+\Delta).
\]
Invariant ramification preserves transversality to a non-invariant
boundary. Therefore
\[
\tang(\cF,\Delta_{\red})=0
\quad\Longrightarrow\quad
\tang(\cH',\Delta_{Y',\red})=0.
\]
The singularities of $Y'$ lie over intersections of invariant
ramification components and are disjoint from $\Supp\Delta_{Y'}$.
Thus $\tau$ is an isomorphism near $\Supp\Delta_{Y'}$. 
Moreover, Proposition~\ref{prop:finite-cover-canonical-pullback},
applied with $B_h=0$, shows that $\tau$ is foliated crepant.
It follows that
\[
K_{\widetilde\cH}=\tau^*K_{\cH'},
\qquad
\widetilde\Delta_Y=\tau^*\Delta_{Y'},
\]
and consequently
\[
K_{\widetilde\cH}+\widetilde\Delta_Y
=
(\pi'\circ\tau)^*(K_{\cF}+\Delta).
\]
We also have
\[
\tang(\widetilde\cH,\widetilde\Delta_{Y,\red})=0.
\]
Since $\pi'\circ\tau$ is generically finite of degree $h$,
\[
\vol(K_{\widetilde\cH}+\widetilde\Delta_Y)
=
h\,\vol(K_{\cF}+\Delta).
\]
In particular, $K_{\widetilde\cH}+\widetilde\Delta_Y$ is big.
By the definition of the canonical cyclic cover, there exists
$0\neq\omega\in H^0(Y,K_{\cH})$ such that
\[
\omega^{\otimes h}=\pi^*s.
\]
Thus the first plurigenus of $\cH$ is nonzero.  Let
\[
\psi\colon(\widetilde Y,\widetilde\cH)\longrightarrow(Y_{\min},\cH_{\min})
\]
be a relatively minimal reduction.  By birational invariance of the
first plurigenus for reduced foliations,
$\delta_1(\cH_{\min})=1$.
It follows from \cite[Lemma~8.1]{Bru15} that
\[
K_{\cH_{\min}}\sim0.
\]

By Lemma~\ref{lem:tangency-free-volume}, we may replace
$(\widetilde Y,\widetilde{\cH},\widetilde\Delta_Y)$ by a
cluster-adapted model over $(Y_{\min},\cH_{\min})$ without changing
its adjoint volume. Lemma~\ref{lem:numerically-trivial-case} then gives
\[
\vol(K_{\widetilde\cH}+\widetilde\Delta_Y)\geq\frac14.
\]
Combining this with the covering formula yields
\[
h\,\vol(K_{\cF}+\Delta)
=
\vol(K_{\widetilde\cH}+\widetilde\Delta_Y)
\geq\frac14.
\]
Therefore
\[
\vol(K_{\cF}+\Delta)
\geq
\frac1{4h}
=
\frac1{4\delta_1(\cF)}.
\]
\end{proof}

Pereira's classification of the possible values of the height turns
the preceding index-dependent estimate into a uniform one.
\begin{corollary}\label{cor:kappa-zero-uniform-bound}
If $\kappa(\cF)=0$, then
\[
\vol(K_{\cF}+\Delta)\geq\frac1{48}.
\]
If, moreover, $X$ is not rational or $\cF$ admits a rational first
integral, then
\[
\vol(K_{\cF}+\Delta)\geq\frac1{24}.
\]
\end{corollary}
\begin{proof}
By \cite[Theorem~1]{Per05},
\[
\delta_1(\cF)\in\{1,2,3,4,5,6,8,10,12\},
\]
so $\delta_1(\cF)\leq12$.
Proposition~\ref{prop:kappa-zero-index-bound}
therefore gives
\[
\vol(K_{\cF}+\Delta)
\geq\frac1{4\delta_1(\cF)}
\geq\frac1{48}.
\]
The additional assumptions are birationally invariant.  Thus, if $X$
is not rational or $\cF$ admits a rational first integral, then
\[
\delta_1(\cF)\in\{1,2,3,4,6\}.
\]
Hence $\delta_1(\cF)\leq6$, and consequently
\[
\vol(K_{\cF}+\Delta)\geq\frac1{24}.
\]
\end{proof}

\subsection{Kodaira dimension one}
\label{subsec:kappa-one-volume-bound}
For Kodaira dimension one, no covering construction is needed.  The
positive part of $K_{\cF_0}$ has numerical dimension one, and the
bigness of $K_{\cF_0}+\Delta_0$ forces it to meet at least one boundary
component positively.  The next lemma makes this intersection
effective in terms of $\delta_2(\cF)=\delta_2(\cF_0)$.
\begin{lemma}\label{lem:positive-intersection-nu-one}
Let $\cF_0$ be a relatively minimal foliation on a smooth projective
surface $X_0$ with $\kappa(\cF_0)=1$, and write
\[
K_{\cF_0}=P+N.
\]
If $K_{\cF_0}+\Delta_0$ is big, then $P\cdot C_{0,i}>0$ for some component
$C_{0,i}$ of $\Delta_0$. Moreover, every such component satisfies
\[
P\cdot C_{0,i}\geq\frac1m
\]
for every integer $m\geq1$ such that
\[
h^0(X_0,mK_{\cF_0})\geq2.
\]
\end{lemma}
\begin{proof}
Since $\kappa(\cF_0)=1$, we have $\nu(\cF_0)=1$, which implies $P\not\equiv0$ and $P^2=0$.  If
$P\cdot C_{0,j}=0$ for every component of $\Delta_0$, then
\[
  (K_{\cF_0}+\Delta_0)\cdot P
  =P^2+N\cdot P+\Delta_0\cdot P=0.
\]
This contradicts the bigness of $K_{\cF_0}+\Delta_0$.
Fix $m$ with $h^0(mK_{\cF_0})\geq2$ and write
\[
  |mK_{\cF_0}|=|M|+Z,
\]
where $M$ is the movable part and $Z$ is the fixed part.  A general
member $M_0\in|M|$ is a nonzero integral nef divisor.  For
$D_m:=M_0+Z\in|mK_{\cF_0}|$, 
\[
  0=P\cdot D_m=P\cdot M_0+P\cdot Z.
\]
Hence nefness of $P$ gives $P\cdot M_0=0$.  The Hodge index theorem now yields
\[
  M_0^2=0,
  \qquad
  M_0\equiv aP,\qquad a>0.
\]
Write
\[
  D_m=P_m+N_m
\]
for the Zariski decomposition of $D_m$.  Since
$D_m\sim mK_{\cF_0}$, one has
\[
  P_m\equiv mP.
\]
Moreover, $M_0$ is nef and $D_m-M_0=Z\geq0$.  The maximality of the
positive part of the Zariski decomposition
(cf.~\cite[Lemma~2.2]{Sak84Anti}) therefore gives
\[
  M_0\leq P_m.
\]
Intersecting with an ample divisor $H$, we obtain
\[
aP\cdot H=M_0\cdot H\leq P_m\cdot H=mP\cdot H.
\]
Since $P\cdot H>0$, this implies $a\leq m$.
Let $C_{0,i}$ be any component with $P\cdot C_{0,i}>0$.
Then
$M_0\cdot C_{0,i}=aP\cdot C_{0,i}$ is a positive integer, so
\[
  P\cdot C_{0,i}
  =\frac{M_0\cdot C_{0,i}}{a}
  \geq\frac1a
  \geq\frac1m.
\]
\end{proof}
Combining this intersection estimate with the nonnegative cluster
formula gives the required adjoint volume bound.
\begin{proposition}\label{prop:kappa-one-volume-bound}
If $\kappa(\cF)=1$, then
\[
  \vol(K_{\cF}+\Delta)
  \geq
  \frac{3}{4\delta_2(\cF)}
  \geq
  \frac1{56}.
\]
\end{proposition}
\begin{proof}
Set $m:=\delta_2(\cF)=\delta_2(\cF_0)$.
By Lemma~\ref{lem:positive-intersection-nu-one}, there is a component
$C_{0,i}$ of $\Delta_0$ such that
\[
  P\cdot C_{0,i}\geq\frac1m.
\]
Since $a_i\in[1/2,1)$, we have $a_i(2-a_i)\geq3/4$.
The cluster formula and the nonnegativity of its local terms therefore
give
\[
  \vol(K_{\cF}+\Delta)
  \geq\sV(\cF_0,\Delta_0)
  \geq a_i(2-a_i)P\cdot C_{0,i}
  \geq\frac{3}{4m}.
\]
Finally, \cite[Theorem~4.14]{CLNP22} gives
$m=\delta_2(\cF)\leq42$, and hence
\[
  \vol(K_{\cF}+\Delta)\geq\frac1{56}.
\]
\end{proof}

\subsection{Kodaira dimension two}
\label{subsec:kappa-two-volume-bound}
The big case requires no cluster analysis: the boundary can only
increase the volume.
\begin{proposition}\label{prop:kappa-two-volume-bound}
If $\kappa(\cF)=2$, then
\[
  \vol(K_{\cF}+\Delta)
  \geq
  \frac{1}{
    \delta_2(\cF)^2
    \bigl(1+\delta_2(\cF)\bigr)
  }.
\]
\end{proposition}
\begin{proof}
Since $K_{\cF}$ is big and $\Delta\geq0$, monotonicity of volume gives
\[
  \vol(K_{\cF}+\Delta)
  \geq
  \vol(K_{\cF}).
\]
The canonical-volume estimate of \eqref{equ:volume-ps-index} gives
\[
  \vol(K_{\cF})
  \geq
  \frac{1}{
    \delta_2(\cF)^2
    \bigl(1+\delta_2(\cF)\bigr)
  },
\]
and the assertion follows.
\end{proof}
\begin{proof}[Proof of Theorem~\ref{thm:adjoint-volume-bounds}]
The cases $\kappa(\cF)=0$ and $1$ follow from
Propositions~\ref{prop:kappa-zero-index-bound} and
\ref{prop:kappa-one-volume-bound}, respectively.  The uniform
constants are supplied by Corollary~\ref{cor:kappa-zero-uniform-bound}
and the bound $\delta_2(\cF)\leq42$ used in
Proposition~\ref{prop:kappa-one-volume-bound}.  The case
$\kappa(\cF)=2$ is Proposition~\ref{prop:kappa-two-volume-bound}.
\end{proof}
\subsection{Application to birational automorphism groups}
We now return to the quotient construction.  The covering formula
identifies the canonical volume upstairs with the adjoint volume of the
tangency-free quotient pair, so the preceding estimates translate
directly into bounds for the birational automorphism group.
\begin{proof}[Proof of Theorem~\ref{thm:main}]
Let $(Y,\cG)$ be a normal projective model of the birational quotient
$(X,\cF)/G$, and let $X'$ be the normalization of $Y$ in
$\mathbb C(X)$.  Then
\[
  q\colon(X',\cF')\longrightarrow(Y,\cG)
\]
is a finite Galois morphism with Galois group $G$, where $\cF'$ is the
saturated transform of $\cF$.  Apply
Corollary~\ref{cor:Galois-cover} to $q$.  It gives a smooth tangency-free quotient pair
$(\bar Y,\bar\cG,\bar\Delta)$ and a smooth birational model
$(\widetilde X,\widetilde\cF)$ of $(X',\cF')$ such that
\begin{equation}
  \vol(K_{\widetilde\cF})
  =|G|\,\vol(K_{\bar\cG}+\bar\Delta).
  \label{eq:main-theorem-covering-volume}
\end{equation}
Since $\widetilde\cF$ is birational to the original canonical
foliation $\cF$, the birational invariance of the canonical volume gives
\[
  \vol(K_{\widetilde\cF})=\vol(\cF).
\]
Moreover,
\[
  \kappa(\bar\cG)=\kappa(\cG),
  \qquad
  \delta_r(\bar\cG)=\delta_r(\cG),
  \qquad
  \tang(\bar\cG,\bar\Delta_{\red})=0,
\]
and $K_{\bar\cG}+\bar\Delta$ is big by
\eqref{eq:main-theorem-covering-volume}.  Moreover, the components of
$\bar\Delta$ are non-$\bar\cG$-invariant and
$\bar\Delta$ has standard coefficients:
\[
  \bar\Delta
  =\sum_i\left(1-\frac1{n_i}\right)B_i,
  \qquad n_i\geq2.
\]
Theorem~\ref{thm:adjoint-volume-bounds} therefore applies to the
quotient pair.
If $\kappa(\cG)=0$, then
\[
  \vol(K_{\bar\cG}+\bar\Delta)
  \geq
  \frac1{4\delta_1(\cG)}.
\]
Equation~\eqref{eq:main-theorem-covering-volume} therefore gives
\[
  |G|
  \leq
  4\delta_1(\cG)\,\vol(\cF).
\]
If $\kappa(\cG)=1$, Theorem~\ref{thm:adjoint-volume-bounds} gives
\[
  \vol(K_{\bar\cG}+\bar\Delta)
  \geq
  \frac3{4\delta_2(\cG)},
\]
and hence
\[
  |G|
  \leq
  \frac43\delta_2(\cG)\,\vol(\cF).
\]
Finally, if $\kappa(\cG)=2$, then
\[
  \vol(K_{\bar\cG}+\bar\Delta)
  \geq
  \frac{1}{
    \delta_2(\cG)^2
    \bigl(1+\delta_2(\cG)\bigr)
  },
\]
which yields
\[
  |G|
  \leq
  \delta_2(\cG)^2
  \bigl(1+\delta_2(\cG)\bigr)
  \,\vol(\cF).
\]
\end{proof}
\begin{proof}[Proof of Corollary~\ref{cor:uniform-automorphism-bounds}]
If $\kappa(\cG)=0$, then $\delta_1(\cG)\leq12$; if
$\kappa(\cG)=1$, then $\delta_2(\cG)\leq42$.  The first two cases of
Theorem~\ref{thm:main} therefore give
\[
  |G|\leq48\,\vol(\cF)
  \qquad\text{and}\qquad
  |G|\leq56\,\vol(\cF),
\]
respectively.
In either case $G$ is nontrivial, since otherwise $\cG=\cF$ would have
Kodaira dimension two.  Hence $|G|\geq2$.  Combining this with
\eqref{eq:main-theorem-covering-volume} and the uniform adjoint-volume
bounds gives
\[
  \vol(\cF)\geq\frac1{24}
  \qquad\text{or}\qquad
  \vol(\cF)\geq\frac1{28},
\]
according as $\kappa(\cG)=0$ or $1$.
\end{proof}

\section*{Acknowledgments}

The author would like to express his sincere gratitude to Professors
Sheng-Li Tan, Mao Sheng, Jun Lu, and Xin L{\"u} for their longstanding
encouragement and support.

\end{document}